\documentclass[10pt,a4paper,reqno]{amsart}
\usepackage{amsthm}
\usepackage{amsmath}
\usepackage{amssymb}
\usepackage{mathtools}
\mathtoolsset{showonlyrefs=true}

\usepackage[font=small]{caption}

\usepackage[shortlabels]{enumitem}
\usepackage{xcolor,graphicx}
\usepackage{multirow}

\makeatletter

\theoremstyle{plain}
\newtheorem{thm}{Theorem}
  \theoremstyle{definition}
  
  \theoremstyle{remark}
  \newtheorem{rem}[thm]{Remark}
  \theoremstyle{plain}
  \newtheorem{prop}[thm]{Proposition}
  \theoremstyle{plain}
  \newtheorem{lem}[thm]{Lemma}
  \theoremstyle{plain}
  \newtheorem{cor}[thm]{Corollary}
 \theoremstyle{definition}
  
  \theoremstyle{remark}
  \newtheorem*{rem*}{Remark}

  \theoremstyle{definition}

\usepackage{mathrsfs}

\newtheorem*{question*}{\it{QUESTION}}
\newtheorem*{problem*}{\it{PROBLEM}}

\newcommand{\N}{\mathbb{N}}
\newcommand{\R}{{\mathbb{R}}}
\newcommand{\C}{{\mathbb{C}}}
\newcommand{\Z}{{\mathbb{Z}}}
\newcommand{\dd}{{{\rm d}}}
\newcommand{\ii}{{\rm i}}

\newcommand{\diag}{\mathop\mathrm{diag}\nolimits}
\newcommand{\lspan}{\mathop\mathrm{span}\nolimits}
\newcommand{\Dom}{\mathop\mathrm{Dom}\nolimits}

\renewcommand{\Im}{\mathop\mathrm{Im}\nolimits}
\renewcommand{\Re}{\mathop\mathrm{Re}\nolimits}
\renewcommand{\Im}{\mathop\mathrm{Im}\nolimits}

\newcommand{\Res}{\mathop\mathrm{Res}\nolimits}

\newcommand{\sn}{\mathop\mathrm{sn}\nolimits}
\newcommand{\cn}{\mathop\mathrm{cn}\nolimits}
\newcommand{\dn}{\mathop\mathrm{dn}\nolimits}

\makeatother

\begin{document}

\graphicspath{{Figures/}}

\title[Spectral analysis of non-self-adjoint Jacobi operator]
{Spectral analysis of non-self-adjoint Jacobi operator associated with Jacobian elliptic functions}

\author{Petr Siegl}
\address[Petr Siegl]{
Mathematisches Institut, 
Universit\"{a}t Bern,
Alpeneggstrasse 22,
3012 Bern, Switzerland
\& On leave from Nuclear Physics Institute ASCR, 25068 \v Re\v z, Czech Republic}
\email{petr.siegl@math.unibe.ch}

\author{Franti\v sek \v Stampach}
\address[Franti\v sek \v Stampach]{
	Mathematisches Institut, 
	Universit\"{a}t Bern,
	Alpeneggstrasse 22,
	3012 Bern, Switzerland
	\& Department of Mathematics, 
	Stockholm University,
	Kr\"{a}ftriket 5,
	SE - 106 91 Stockholm, Sweden
}
\email{stampfra@fjfi.cvut.cz}



\date{February 4, 2017}

\begin{abstract}
We perform the spectral analysis of a family of Jacobi operators $J(\alpha)$ depending on a complex parameter $\alpha$.
If $|\alpha|\neq1$ the spectrum of $J(\alpha)$ is discrete and formulas for eigenvalues and eigenvectors are established
in terms of elliptic integrals and Jacobian elliptic functions. If $|\alpha|=1$, $\alpha \neq \pm 1$, the essential spectrum of $J(\alpha)$
covers the entire complex plane. In addition, a formula for the Weyl $m$-function as well as the asymptotic 
expansions of solutions of the difference equation corresponding to $J(\alpha)$ are obtained. Finally, the completeness 
of eigenvectors and Rodriguez-like formulas for orthogonal polynomials, studied previously by Carlitz, are proved.
\end{abstract}

\maketitle

\section{Introduction}

We investigate spectral properties of a one-parameter family of Jacobi operators $J(\alpha)$, $\alpha \in \C$, acting in 
$\ell^{2}(\N)$, with emphasis on obtaining the spectral results in the most explicit form.
The operator $J(\alpha)$ is determined by the semi-infinite Jacobi matrix $\mathcal{J}(\alpha)$ whose diagonal vanishes and off-diagonal sequence $\{w_{n}\}_{n \in  \N}$ is given by
\begin{equation}
 w_{n}=\begin{cases}
        n & \mbox{ for } n \mbox{ odd},\\
        \alpha n & \mbox{ for } n \mbox{ even}.
       \end{cases}
\label{eq:def_seq_w}
\end{equation}
Thus, with respect to the standard basis of $\ell^{2}(\N)$, the matrix $\mathcal{J}(\alpha)$ is of the form
\begin{equation}
 \mathcal{J(\alpha)}=\begin{pmatrix}
 0 & 1\\
 1 & 0 & 2\alpha\\
 & 2\alpha & 0 & 3\\
 & & 3 & 0 & 4\alpha\\
 & & & \ddots & \ddots & \ddots
\end{pmatrix}\!.
\label{eq:def_jacobi_mat}
\end{equation}

The operator $J(\alpha)$ is self-adjoint if and only if $\alpha\in\R$. We focus mainly on the non-self-adjoint case with a general $\alpha\in\C$, although we restrict ourselves to $|\alpha| \leq 1$ in the body of the paper.
For $|\alpha|>1$, the spectral analysis is in all aspects very similar and the main results for this case, omitting the detailed proofs, are summarized in the last section.

We investigate the localization of essential spectrum and eigenvalues, asymptotic properties of eigenvectors, their completeness and possible basisness. 
It turns out that the operator $J(\alpha)$ constitutes one of not many concrete unbounded non-self-adjoint operators 
whose spectral properties can be described explicitly and, in addition, whose spectrum is entirely real for a certain 
(non-real) range of parameter $\alpha$. Furthermore, the Jacobi matrix $\mathcal{J}(\alpha)$ belongs to the class with periodically 
modulated unbounded weights, where the so-called spectral phase transition phenomena has been observed \cite{janas_raot01,naboko_jat09,simonov_otaa07}, 
see also \cite{janas_jfa02}, however, all in the self-adjoint setting.  
To our best knowledge, except for the classical example of a perturbed shift operator in $\ell^2(\Z)$ and various versions of it, see \cite[Ex.~IV.3.8]{kato66},  $J(\alpha)$ is the first instance of a non-self-adjoint (unbounded) Jacobi operator with the transition property, namely with a sudden and complete change of the spectral character when the parameter $\alpha$ crosses the  unit circle.

The vast literature on specific families of self-adjoint Jacobi operators shows that spectral properties are usually closely related with special functions. In our case, the major role is played by elliptic integrals and Jacobian elliptic functions, where the parameter $\alpha$ enters as a (complex) modulus. Moreover, the spectral analysis of a Jacobi operator can be reformulated to the study of specific properties of corresponding family of orthogonal polynomials, see \cite{akhiezer90}. The family associated with $\mathcal{J}(\alpha)$ does not belong to the Askey-scheme, which can serve, see \cite{koekoek10}, as a rich source of Jacobi operators with explicitly solvable spectral problem, however, it was studied before by Carlitz in \cite{carlitz_dmj60,carlitz_dmj61} for $\alpha \in (0,1)$.

In Section \ref{sec:gen_prop}, we introduce the unique closed and densely defined Jacobi operator $J(\alpha)$ associated with the matrix $\mathcal{J}(\alpha)$ and derive some of its fundamental properties relying on general theorems of spectral and perturbation theory for linear operators.
It is shown that the residual spectrum of $J(\alpha)$ is empty, the resolvent of $J(\alpha)$ is compact if $|\alpha|<1$, and, on the other hand, the essential spectrum of $J(\alpha)$ is non-empty if $|\alpha|=1$.

Section \ref{sec:sa_case} is devoted to the self-adjoint case, i.e., for $\alpha\in\R$. We start with a simple algebraic identity, which might be deduced from a continued fraction formula going back to Stieltjes, and obtain a formula for the Fourier transform of the spectral measure. This yields the spectrum of $J(\alpha)$ immediately. Moreover,
a suitably applied Laplace transform enables us to derive the Mittag-Leffler expansion for the Weyl $m$-function.

Main results are derived within Section \ref{sec:nsa_case} where the non-self-adjoint case is treated. If $|\alpha|< 1$, we obtain expressions for eigenvalues of $J(\alpha)$, integral formulas for eigenvectors and their asymptotic expansions for the index going to infinity. Moreover, the set of eigenvectors is shown to be complete in $\ell^{2}(\N)$. For $|\alpha|=1$ and $\alpha\neq\pm1$, we prove by constructing singular sequences that the essential spectrum of $J(\alpha)$ coincides with all of $\C$. In addition, a Rodriguez-like formula for the associated orthogonal polynomials is derived as well as certain generating function formulas for
quantities closely related to eigenvectors. 
The question whether the set of eigenvectors forms the Riesz (or Schauder) basis remains open, nonetheless, the authors
incline to the negative answer. A numerical analysis of pseudospectra, supporting the opinion, is presented and a formula for the norm of eigenprojections, which might be useful in excluding the basisness of eigenvectors, is established.

Finally, Section \ref{sec:case_alp_geq_1} contains a brief summary of corresponding results for $|\alpha|>1$ and the paper is concluded by appendices on selected properties of Jacobian elliptic functions and numerical analysis of pseudospectra of $J(\alpha)$.

\section{General properties of $J(\alpha)$} \label{sec:gen_prop}

Recall that, within the standard construction of a linear operator associated with matrix $\mathcal{J}(\alpha)$, one defines the couple of operators $J_{\min}(\alpha)$ and $J_{\max}(\alpha)$, see, for example, \cite[Sec.~2]{beckermann_jcam01}. 
In more detail, for $x \in \ell^2(\N)$, which is to be understood as semi-infinite column vector in the following, $\mathcal{J}(\alpha)x$ is given by the formal matrix multiplication. 
The minimal operator $J_{\min}(\alpha)$ is defined as the operator closure of an auxiliary operator $J_{0}(\alpha)$, 
\[
J_{0}(\alpha) x = \mathcal{J}(\alpha)x, \quad \Dom(J_{0}(\alpha))=\lspan\{e_n\mid n\in\N\},
\]
where $e_n$ stands for the $n$th vector of the standard basis of $\ell^{2}(\N)$; $J_{0}(\alpha)$ can be shown to be always closable. The maximal operator $J_{\max}(\alpha)$ is defined as
\[
J_{\max}(\alpha) x = \mathcal{J}(\alpha)x, \quad \Dom(J_{\max}(\alpha))=\left\{x\in\ell^{2}(\N) \mid J_{\max}(\alpha)x\in\ell^{2}(\N)\right\}\!.
\]
Clearly, $J_{\min}(\alpha)\subset J_{\max}(\alpha)$, however, $J_{\min}(\alpha)=J_{\max}(\alpha)$ in our case.
This equality is guaranteed by Carleman's sufficient condition \cite[Ex.~2.7]{beckermann_jcam01}:
\[
 \sum_{n=1}^{\infty}\frac{1}{|w_n|}=\infty,
\]
which holds for $w_n$ given by \eqref{eq:def_seq_w} if $\alpha\neq0$. 
In addition, one has $J_{\rm min}(\alpha)^{*}=J_{\max}(\overline{\alpha})$ for all $\alpha\in\C$. 
The situation for $\alpha=0$ is somewhat special but trivial and the equality $J_{\min}(0)=J_{\max}(0)$ 
remains true as well. Thus, the subscripts $\min$ and $\max$ can be omitted and the unique Jacobi operator
determined by $\mathcal{J}(\alpha)$ is denoted by $J(\alpha)$.

Let us summarize these facts in the following proposition.

\begin{prop}
 For all $\alpha\in\C$, Jacobi matrix \eqref{eq:def_jacobi_mat} determines the unique unbounded Jacobi operator $J(\alpha)$,
 for which it holds $J(\alpha)^{*}=J(\overline{\alpha})$. Consequently, operator $J(\alpha)$ is $C$-self-adjoint, i.e. 
 $J(\alpha)^{*}=CJ(\alpha)C$, where $C$ is the complex conjugation operator on $\ell^{2}(\N)$.
\end{prop}

In the next corollary, we summarize several general spectral properties of $J(\alpha)$ that follow immediately 
from its $C$-self-adjointness, see e.g.~\cite[Sec.~III.5, IX.1, Thm.~IX.1.6]{edmunds87} and \cite[Cor.~2.1]{borisov_ieot08}. 
Notice that there are several definitions of essential spectra for non-self-adjoint operators, here we follow the 
notations of \cite[Sec.~IX.1]{edmunds87}. In this paper, we work with $\sigma_{e2}$, which can be characterized 
by singular sequences, see \cite[Def.~IX.1.2, Thm.~IX.1.3]{edmunds87}.
\begin{cor}
For all $\alpha\in\C$, the residual part of the spectrum of $J(\alpha)$ is empty and four definitions of essential spectra (see \cite[Sec.~IX.1]{edmunds87}) coincide, namely 
$$\sigma_{e1}(J(\alpha)) =\sigma_{e2}(J(\alpha))=\sigma_{e3}(J(\alpha))=\sigma_{e4}(J(\alpha)).$$
\end{cor}

Note that operators $J(\alpha)$ and $J(-\alpha)$ are unitarily equivalent via the unitary operator
$U=\diag(1,1,-1,-1,1,1,-1,-1,\dots)$. Hence, if spectral  properties of $J(\alpha)$ are investigated, 
the range of $\alpha$ can be restricted to a half-plane, for example $\Re \alpha\geq0$.

Next, $J(\alpha)$ has a compact resolvent if $|\alpha|<1$, but it is not 
the case if $|\alpha|=1$.

\begin{prop}\label{prop:J_inv}
	The following statements hold true.
\begin{enumerate}[{\upshape i)}]
	\item If $|\alpha|<1$, then $0 \in \rho(J(\alpha))$ and $J(\alpha)^{-1}$ is a Hilbert-Schmidt operator.
\item If $|\alpha|=1$, then $0 \in \sigma_{e2}(J(\alpha))$.
\end{enumerate}
\end{prop}
\begin{proof}
The verification of the statement (i) is trivial for $\alpha=0$. Further we assume $\alpha\neq0$.

Denote by $\{u_n\}_{n\geq1}$ and $\{v_n\}_{n\geq1}$ the two solutions of the second-order difference equation
\begin{equation}\label{dif.eq.J0}
w_{n-1}y_{n-1}+w_{n}y_{n+1}=0, \quad n\geq2,
\end{equation}
determined by the initial values $u_{1}=1$, $u_{2}=0$  and  $v_{1}=0$, $v_{2}=1$.
A straightforward computation leads to formulas
\begin{align}
u_{2n}&=0,& u_{2n+1} &= (-1)^n \alpha^{-n} \frac{(2n-1)!!}{(2n)!!} 
= (-1)^n \alpha^{-n} \frac{1}{4^n} \binom{2n}{n},
\label{eq:eigenvec_0}
\\
v_{2n+1}&=0,& v_{2n+2}&=(-1)^{n}\alpha^{n}\frac{(2n)!!}{(2n+1)!!}
= (-1)^{n}\alpha^{n}\frac{4^n}{n+1}\frac{1}{\binom{2n+1}{n}}, & n \in \N.
\end{align}
Clearly, the matrix $R$ with elements
\[
R_{j,k}=\begin{cases} u_{j}v_{k}, & \quad  1\leq j\leq k,\\
u_{k}v_{j}, & \quad  1\leq k\leq j,
\end{cases}
\]
is the formal inverse to $J(\alpha)$. Substituting the explicit expressions for $u$ and $v$ into the last formula,
we get (where $(-1)!!=0!!=1$ by convention)
\begin{align*}
R_{2m+1,2n+2}&=(-1)^{m+n}\alpha^{n-m}\frac{(2m-1)!!}{(2m)!!}\frac{(2n)!!}{(2n+1)!!}, \quad 0\leq m\leq n,
\\
R_{2m+2,2n+1}&=(-1)^{m+n}\alpha^{m-n}\frac{(2n-1)!!}{(2n)!!}\frac{(2m)!!}{(2m+1)!!}, \quad 0\leq n\leq m;
\end{align*}
all other entries vanish.

Proof of the statement (i): If $|\alpha|<1$, then matrix $R$ represents a bounded, even a Hilbert-Schmidt operator on
$\ell^{2}(\mathbb{N})$, thus this operator coincides with $J(\alpha)^{-1}$.
Indeed, for the Hilbert-Schmidt norm of $R$, one has
\[
\|R\|_2^{2}=
2\sum_{m,n\geq0}|R_{2m+1,2n+2m+2}|^{2}=2\sum_{m,n\geq0}|\alpha|^{2n}\left[\frac{(2m-1)!!}{(2m)!!}\frac{(2n+2m)!!}{(2n+2m+1)!!}\right]^{2}
\]
and the expression in the squared brackets can be rewritten as
\begin{equation*}
\frac{4^{n}}{2m+2n+1}\binom{2m}{m}\bigg/\binom{2m+2n}{m+n}.
\end{equation*}
The well-known bound for the central binomial coefficient 
\begin{equation} \label{bin.bound}
\frac{4^{n}}{2\sqrt{n}}\leq\binom{2n}{n}\leq\frac{4^{n}}{\sqrt{3n+1}}, \quad n\in\mathbb{N},
\end{equation}
and further elementary estimates show that $\|R\|_2<\infty$.

Proof of the statement (ii): If $|\alpha|= 1$, we construct a singular sequence for $J(\alpha)$,
see \cite[Def.~IX.1.2, Thm.~IX.1.3]{edmunds87}.
For $a \in (0,1)$, we define sequences $u(a)$ with entries
\begin{equation}\label{una.def}
u_n(a)= a^n u_n, \quad n \in \N,
\end{equation}
where $u_n$ are as in \eqref{eq:eigenvec_0}. It follows from \eqref{bin.bound} that $u(a) \in \ell^2(\N)$ for all $a \in (0,1)$ and moreover
\begin{equation}\label{ua.lb}
\|u(a)\|^2 
= 
\sum_{n = 0}^\infty a^{4n+2} \left( \frac{1}{4^n} \binom{2n}{n} \right)^2
\geq
\frac{a^2}4 \sum_{n = 1}^\infty \frac{a^{4n}}{n} 
= 
-\frac{a^2}4  \ln\left(1-a^4\right).
\end{equation} 
On the other hand, since $u$ is the solution of the difference equation \eqref{dif.eq.J0}, we get 
\begin{align*}
(\mathcal{J}(\alpha)u(a))_{2n-1}&=0,& &
\\
(\mathcal{J}(\alpha)u(a))_{2n}&= w_{2n-1} a^{2n-1} u_{2n-1} + w_{2n} a^{2n+1} u_{2n+1} &&
\\& =-w_{2n} a^{2n-1}(1-a^2)u_{2n+1},& n \in \N.&
\end{align*}
Hence $u(a)\in\Dom J(\alpha)$ for all $a\in(0,1)$, and, using \eqref{bin.bound} again, we obtain
\begin{equation}\label{J0.ub}
\begin{aligned}
\|J(\alpha)u(a)\|^2  &= \frac{4(1-a^2)^2}{a^2} \sum_{n = 1}^\infty n^2 a^{4n} \left( \frac{1}{4^n} \binom{2n}{n} \right)^2
\\
& 
\leq\frac{2(1-a^2)^2}{a^2}  \sum_{n = 1}^\infty n a^{4n} = \frac{2a^2}{(1+a
	^2)^2}.
\end{aligned}
\end{equation}
By putting \eqref{una.def}, \eqref{ua.lb} and \eqref{J0.ub} together, we receive
\begin{equation}
\forall k \in \N, \quad \lim_{a \to 1-}\frac{\langle e_k, u(a)\rangle}{\|u(a)\|} = 0 
\quad \text{and} \quad
\lim_{a \to 1-} \frac{\|J(\alpha)u(a)\|}{\|u(a)\|} = 0,
\end{equation}
thus $0 \in \sigma_{e2}(J(\alpha))$.
\end{proof}

For $|\alpha|<1$, the operator $J(\alpha)$ can be viewed as a perturbation of $J(0)$ with the
relative bound smaller than $1$. For later purposes we formulate the following lemma.

\begin{lem}\label{lem:pert}
Let $\alpha \in \C$, $|\alpha|<1$. Then
\begin{enumerate}[\upshape (i)]
	\item for every $\varepsilon>0$, there exists $C(\varepsilon)>0$ such that, for all $u \in \Dom(J(0))$,
	\begin{equation}\label{J.alph.rb.1}
	\|(J(\alpha)-J(0))u\|
	\leq (1+\varepsilon)|\alpha|  \|J(0)u\| + C(\varepsilon)\|u\|,
	\end{equation}
	
	\item \label{prop:J.com.res.an} for every $r \in (0,1)$, there exists a non-empty open set $\mathcal Z_r \subset \C$ with $\mathcal Z_r \cap \R = \emptyset$ such that, for all $z \in  \mathcal Z_r$ and all $\alpha \in B_r(0)$,  $(J(\alpha)-z)^{-1}$ exists and it is a holomorphic bounded-operator-valued function of $\alpha$ on $B_r(0)$.
\end{enumerate}
\end{lem}
\begin{proof}
Let $\alpha \neq 0$ and denote by $M$ the operator $ \frac 1\alpha (J(\alpha)-J(0))$; notice that $M$ is independent of $\alpha$.
For every $u \in \Dom(J(0))$,
\begin{equation}\label{Jp.ub}
\begin{aligned}
\|\alpha M u\|^2 & = \sum_{n=1}^\infty 
\left(
|w_{2n} u_{2n+1}|^2 + |w_{2n} u_{2n}|^2
\right)
\leq
|\alpha|^2 \sum_{n=1}^\infty n^2 |u_n|^2.
\end{aligned}
\end{equation}
On the other hand, for every $u \in \Dom(J(0))$,
\begin{equation}
\begin{aligned}
\|J(0)u\|^2 & = \sum_{n=1}^\infty 
\left(
|w_{2n-1} u_{2n}|^2 + |w_{2n-1} u_{2n-1}|^2
\right)
\\&=
\sum_{n=1}^\infty 
\left(
(2n)^2 |u_{2n}|^2\left(1-\frac{1}{2n}\right)^2  + (2n-1)^2|u_{2n-1}|^2 
\right)\!,
\end{aligned}
\end{equation}
hence for every $\delta>0$, there exists $\tilde C(\delta)>0$ such that 
\begin{equation}\label{J0.lb}
\begin{aligned}
\|J(0)u\|^2  &\geq \left(1-\delta \right) \sum_{n=1}^\infty n^2 |u_n|^2  - \tilde C(\delta)  \|u\|^2.
\end{aligned}
\end{equation}
By putting \eqref{Jp.ub}, \eqref{J0.lb} together and using Young inequality, we obtain the statement (i).

Proof of the statement (ii): 
Notice that $J(0)=J(0)^*$, thus $\|(J(0)-z)^{-1}\|\leq \frac{1}{|\Im z|}$ and $\|J(0)(J(0)-z)^{-1}\| \leq \frac{|z|}{|\Im z|}$ 
for $z \notin \R$. Further, for any $z \notin \R$ and $u \in \ell^2(\N)$, we have from \eqref{J.alph.rb.1} that
\begin{equation}\label{M.bound}
\begin{aligned}
\|\alpha M (J(0)-z)^{-1}u\| &\leq (1+\varepsilon)|\alpha| \|J(0)(J(0)-z)^{-1} u\| + C(\varepsilon) \|(J(0)-z)^{-1} u\|
\\
& \leq \left( \frac{(1+\varepsilon)|\alpha||z|}{|\Im z|} + \frac{C(\varepsilon)}{|\Im z|}\right) \|u\|,
\end{aligned}
\end{equation}
where $\varepsilon>0$ is arbitrary. If $r<1$, then we can clearly select $\varepsilon >0$ such that $(1+\varepsilon)r<1$. Therefore there exists a non-empty open set $\mathcal{Z}_r \subset \C$ with $\mathcal Z_r \cap \R = \emptyset$ such that, for all $\alpha\in B_{r}(0)$ and all $z \in \mathcal Z_r$,
\begin{equation}
\|\alpha M (J(0)-z)^{-1}\| <1.
\end{equation}
Hence we have the standard representation of resolvent of $J(\alpha)$ based on Neumann series
\begin{equation}\label{Jres.repr}
\begin{aligned}
(J(\alpha)-z)^{-1} & = (J(0)-z)^{-1} (I + \alpha M(J(0)-z)^{-1}))^{-1} 
\\ &= (J(0)-z)^{-1} \sum_{n=0}^{\infty} (-\alpha)^n (M(J(0)-z)^{-1}))^n, 
\end{aligned}
\end{equation}
from which the analyticity in $\alpha$ follows.
\end{proof}

\section{The self-adjoint case}\label{sec:sa_case}

In this section, we analyze the spectral properties of $J(\alpha)$ for $0\leq\alpha\leq1$. Some of the following
results may be deduced from the properties of orthogonal polynomials studied in \cite[\S\S~7]{carlitz_dmj60}. Here
we provide an independent brief derivation based exclusively on techniques developed for spectral analysis of Jacobi operators,
see \cite{teschl00}.

\subsection{Preliminaries}

Let us start by the amazing formula 
\begin{equation}
 \int_{0}^{\infty}e^{-u}\cn(zu,\alpha)\mbox{d}u=
\dfrac{1}{1+\dfrac{z^{2}w_{1}^{2}\phantom{I}}{1+\dfrac{z^{2}w_{2}^{2}\phantom{I}}{1+\dfrac{z^{2}w_{3}^{2}\phantom{I}}{1+\ldots}}}}
\label{eq:S-frac}
\end{equation}
which goes back to Stieltjes, see \cite{stieltjes93}. This identity is to be understood as the equality between two elements
of the ring of formal power series in the indeterminate $z$. The formal power series for the formal Laplace transform on the LHS
of \eqref{eq:S-frac} equals
\[
 \sum_{n=0}^{\infty}(-1)^{n}C_{2n}\left(\alpha^{2}\right)z^{2n},
\]
as one deduces with the aid of \eqref{eq:cn_tayl}. On the other hand, the coefficients of the power series associated with 
the Stieltjes continued fraction on the RHS of \eqref{eq:S-frac} is known to be expressible in terms of the first 
diagonal element of an integer power of the Jacobi matrix $\mathcal{J}(\alpha)$. This can be deduced, for example, from the 
Stieltjes' Expansion Theorem \cite[Thm.~53.1]{wall48}; see also \cite{flajolet_dm80} for more details.
Namely, the RHS of \eqref{eq:S-frac} equals
\[
 \sum_{n=0}^{\infty}(-1)^{n}\left(\mathcal{J}(\alpha)^{2n}\right)_{1,1}z^{2n}.
\]
In addition, since the diagonal of $\mathcal{J}(\alpha)$ vanishes, one has $\left(\mathcal{J}(\alpha)^{2n+1}\right)_{1,1}=0$ for all $n \in \N_0$.

Consequently, formula \eqref{eq:S-frac} yields identities
\begin{equation}
 \langle e_{1}, J(\alpha)^{2n+1}e_{1}\rangle=0, \quad \quad \langle e_{1}, J(\alpha)^{2n}e_{1}\rangle=C_{2n}\left(\alpha^{2}\right), \quad n\in \N_0, \ \alpha\in\C.
 \label{eq:powJ_eq_C}
 \end{equation}
 
\subsection{Spectrum in the case $0\leq\alpha\leq1$}

According to equalities \eqref{eq:cn_tayl} and \eqref{eq:powJ_eq_C}, 
function $\cn(z,\alpha)$ can be written as
\[
 \cn(z,\alpha)=\sum_{n=0}^{\infty}\frac{(-1)^{n}}{(2n)!}z^{2n}\langle e_1, J(\alpha)^{2n}e_1\rangle.
\]
Define $\mu(\cdot):=\langle e_{1}, E_{J}(\cdot)e_{1}\rangle$ where $E_{J}$ stands for the spectral 
measure of the self-adjoint operator $J(\alpha)$. 
Then, by the Spectral theorem, we have
\[
  \cn(z,\alpha)=\sum_{n=0}^{\infty}\frac{(-1)^{n}}{(2n)!}z^{2n}\int_{\mathbb{R}}x^{2n}\mbox{d}\mu(x).
\]
Since the power series of $\cn(z,\alpha)$ converges absolutely for $|z|<\pi/2$, Fubini's theorem
justifies the interchange of the sum and the integral and, with the help of \eqref{eq:powJ_eq_C}, we get
\begin{equation}
 \int_{\mathbb{R}}e^{ixz}\mathrm{d}\mu(x)=\cn(z,\alpha),
\label{eq:expizJ}
\end{equation}
which is true in the circle $|z|<\pi/2$. Nevertheless, since 
\[
 \int_{\mathbb{R}}e^{a|x|}\mathrm{d}\mu(x)<\infty
\]
for all $0<a<\pi/2$, the LHS of \eqref{eq:expizJ} is a function analytic in the strip $|\Im z|<\pi/2$
and formula \eqref{eq:expizJ} remains true for all $z\in\C$, $|\Im z|<\pi/2$. One can show that the
largest strip where \eqref{eq:expizJ} holds is in fact $|\Im z|<K'(\alpha)$; here $K'(\alpha)$ is the conjugate elliptic integral, see Appendix~\ref{app:elliptic}.

The LHS of \eqref{eq:expizJ} is nothing but the Fourier transform of $\mu$, i.e., $\mathcal{F}[\mu](z)=\cn(z,\alpha)$,
where the measure $\mu$ is identified with the corresponding tempered distribution.
Consequently, by the inverse Fourier transform to the function $\cn(z,\alpha)$, 
we recover the spectral measure $\mu$.

Recall that, in the distributional sense, one has
\[
 \mathcal{F}^{-1}\left[\cos(ax)\right](t)=\frac{1}{2}\left(\delta(t-a)+\delta(t+a)\right), \quad a, t\in\R.
\]
Hence, taking into account the expansion \eqref{eq:cn_Four_ser}, one computes
\[
\mu(t)=\frac{\pi}{\alpha K(\alpha)}\sum_{n=0}^{\infty}\frac{q^{n+1/2}}{1+q^{2n+1}}\left[\delta\left(t-\frac{(2n+1)\pi }{2K(\alpha)}\right)+\delta\left(t+\frac{(2n+1)\pi }{2K(\alpha)}\right)\right], \ \ t \in \R, \ \alpha \in (0,1),
\]
where $K(\alpha)$ is the complete elliptic integral of the first kind and $q$ is the nome, see Appendix~\ref{app:elliptic}.
The measure $\mu$ coincides with the measure of orthogonality of polynomials studied by Carlitz and the above
formula is in agreement with results of \cite[\S\S~7]{carlitz_dmj60}. Since the support of the measure $\mu$ coincide
with the spectrum of $J(\alpha)$, we get
\[
 \sigma\left(J(\alpha)\right)=\frac{\pi}{2K(\alpha)}\left(2\mathbb{Z}+1\right), \quad \alpha \in [0,1);
\]
the special case $\alpha=0$, for which $K(0)=\pi/2$, can be verified directly.

If $\alpha=1$, then $\cn(z,1)=1/\cosh(z)$. Recall that
\[
  \mathcal{F}^{-1}\left[\frac{1}{\cosh(x)}\right](t)=\frac{1}{2\cosh\left(\pi t/2\right)}, \quad t \in \R.
\]
Thus, starting at the formula \eqref{eq:expizJ}, one concludes that the measure $\mu$ is absolutely 
continuous and its density equals
\begin{equation}
 \frac{\mbox{d}\mu}{\mbox{d}t}=\frac{1}{2\cosh\left(\pi t/2\right)}, \quad t\in\mathbb{R}.
\label{eq:mu_density_alpha=1} 
\end{equation}

We summarize the obtained results in the following proposition.

\begin{prop}\label{prop:sa_spec}
 We have
 \[
  \sigma\left(J(\alpha)\right)=\begin{cases} \displaystyle
                               \frac{\pi}{2K(\alpha)}\left(2\mathbb{Z}+1\right), & \quad  |\alpha|<1,\\[2mm]
                               \mathbb{R}, & \quad |\alpha|=1.
                              \end{cases}
 \]
\end{prop}

We remark that the spectral decomposition of $J(1)$ can be derived with the aid of the special case of Meixner-Pollaczek polynomials
 \cite[Sec.~9.7]{koekoek10}
\[
 M_{n}(x)=\ii^{n}n!\,_{2}F_{1}\left(-n,\frac{1+\ii x}{2};1,2\right)\!, \quad n \in \N_0, \ x \in \R.
\]
These polynomials satisfy the recurrence
$$M_{n+1}(x)=xM_{n}(x)-n^{2}M_{n-1}(x), \quad  n \in \N,$$
with initial conditions $M_{0}(x)=1$ and $M_{1}(x)=x$. Their orthogonality relation reads
\[
 \int_{\mathbb{R}}M_{m}(x)M_{n}(x)\frac{\mbox{d}x}{\cosh(\pi x/2)}=2(n!)^{2}\delta_{m,n}, \quad m,n \in \N_0.
\]
Thus, if we set 
\[
 \phi_{n}(x)=\frac{1}{(n-1)!}\left(2\cosh\left(\frac{\pi x}{2}\right)\right)^{-\frac 12}M_{n-1}(x), \quad n \in \N, \ x \in \R,
\]
and $\varphi_{x}^{T}=\left(\phi_{1}(x),\phi_{2}(x),\dots\right)$, then $J(1)\varphi_{x}=x\varphi_{x}$ for all $x\in\R$
and
\[
 \int_{\mathbb{R}}\phi_{m}(x)\phi_{n}(x)\mbox{d}x=\delta_{m,n}.
\]
Consequently, one can introduce the unitary mapping $U:\ell^{2}(\mathbb{N})\to L^{2}(\mathbb{R},\mbox{d}x)$ by setting $Ue_{n}:=\phi_{n}$, for all $n\in\mathbb{N}$.
Then, clearly
\[
 (U\psi)(x)=\langle\varphi_{x},\psi\rangle_{\ell^{2}}=\sum_{n=1}^{\infty}\phi_{n}(x)\psi_{n}, \quad
\quad
 U^{-1}f=\langle\varphi_{x},f\rangle_{L^{2}}=\int_{\mathbb{R}}f(x)\varphi_{x}\dd x.
\]
Finally, for $x \in \R$, one easily verifies
\[
 UJ(1)U^{-1}\phi_{1}(x)=\phi_{2}(x)=x\phi_{1}(x),
\]
and
\[
 UJ(1)U^{-1}\phi_{n}(x)=(n-1)\phi_{n-1}(x)+n\phi_{n+1}(x)=x\phi_{n}(x), \quad n\geq2.
\]
Thus, $J(1)$ is unitarily equivalent to the multiplication operator by the independent variable acting on $L^{2}(\mathbb{R},\mbox{d}x)$. 
Consequently, we have again $\sigma(J(1))=\sigma_{ac}(J(1))=\mathbb{R}$.

\subsection{Weyl $m$-function in the case $0<\alpha\leq1$}

Recall the Weyl $m$-function of Jacobi operator $J(\alpha)$ is defined as
\begin{equation}
 m(z,\alpha)=\langle e_1,(J(\alpha)-z)^{-1}e_1\rangle, \quad z\in\C, \ \Im z\neq0.
 \label{eq:def_m_fction}
\end{equation}
We derive an explicit formula for the $m$-function, in fact its Mittag-Leffler expansion. The proof  relies on the identity \eqref{eq:expizJ} and the well-known relation 
for a self-adjoint operator $A$ in a Hilbert space $\mathcal H$, see e.g. \cite[Eq.~(VIII.9)]{reed1} or 
\cite[Chp.~5,~Prob.~31.(b)]{blank08}, namely, for all $f\in\mathcal{H}$,
 \begin{equation}
(A-z)^{-1}f=\pm\ii\int_{0}^{\infty}e^{\pm\ii zt}e^{\mp\mathrm{i}At}f\mbox{d}t, \quad  \Im z\gtrless0.
 \label{eq:resolv_eq_Laplace_exp}
 \end{equation}

\begin{prop}
Let the function $m(z,\alpha)$ be as in \eqref{eq:def_m_fction}. If $0<\alpha<1$, then we have
\begin{equation}
m(z,\alpha)=-\frac{\pi}{\alpha K(\alpha)}\sum_{n=-\infty}^{\infty}\frac{q^{n+1/2}}{1+q^{2n+1}} 
	\frac{1}{z-\frac{(2n+1)\pi}{2K(\alpha)}}, \quad \Im z \neq 0, 
\label{eq:m_fction_0<alpha<1}
\end{equation}
and if $\alpha=1$, then we have
 \begin{equation}
 m(z,1)=\pm\frac{\ii}{2}\left(\psi\left(\frac{3}{4}\mp\frac{\ii z}{4}\right)-\psi\left(\frac{1}{4}\mp\frac{\ii z}{4}\right)\right), \quad \Im z \gtrless 0,
 \label{eq:m_fction_alpha=1} 
 \end{equation}
 where $\psi$ stands for the digamma function.
\end{prop}

\begin{proof}
It follows from \eqref{eq:def_m_fction} and the self-adjointness of $J(\alpha)$ that
$
\overline{m(\bar{z},\alpha)}=m(z,\alpha),
$
thus it suffices to consider $z$ with $\Im z>0$. From \eqref{eq:resolv_eq_Laplace_exp} and \eqref{eq:expizJ}, one gets
\begin{equation}
 m(\ii z,\alpha)=\ii\mathcal{L}[\cn(t,\alpha)](z), \quad \Re z>0, \ \alpha \in (0,1],
\label{eq:m_eq_Laplace}
\end{equation}
where $\mathcal{L}$ denotes the Laplace transform.

Let $0<\alpha<1$. Recalling \eqref{eq:cn_Four_ser} together with the elementary formula
\[
 \mathcal{L}[\cos(at)](z)=\frac{z}{a^{2}+z^{2}}, \quad a\in\R, \ \Re z>0,
\]
one computes
\[
 \mathcal{L}[\cn(t,\alpha)](z)=\frac{2\pi z}{\alpha K(\alpha)}\sum_{n=0}^{\infty}\frac{q^{n+1/2}}{1+q^{2n+1}}
 \frac{1}{z^{2}+\frac{(2n+1)^{2}\pi^{2}}{4K(\alpha)^{2}}}, \quad \Re z>0, \ \alpha \in (0,1).
\]
Formula \eqref{eq:m_fction_0<alpha<1} now follows from \eqref{eq:m_eq_Laplace} and the identity above.

If $\alpha=1$, then we have 
\[
 \mathcal{L}\left[\frac{1}{\cosh(t)}\right](z)=\frac{1}{2}\left(\psi\left(\frac{z+3}{4}\right)-\psi\left(\frac{z+1}{4}\right)\right),
 \quad \Re z>-1,
\]
see \cite[Eq.~3.541(6)]{gradshteyn00}, and similarly, by \eqref{eq:m_eq_Laplace}, one arrives at \eqref{eq:m_fction_alpha=1}.
\end{proof}

\begin{rem}
 By the Stieltjes inversion formula, the density of the absolutely continuous part of the spectral measure $\mu$ 
 can be recovered from the Weyl $m$-function as the limit
 \[
  \frac{1}{\pi}\lim_{\epsilon\to0+}\Im m(x+\mathrm{i}\epsilon),
 \]
 see \cite[Chp.~2]{teschl00}. In the case $\alpha=1$, one can compute the limit explicitly and
 re-prove \eqref{eq:mu_density_alpha=1}. Indeed, with the aid of formula \cite[Eq.~6.3.16]{abramowitz64}
 \[
  \psi(z)=-\gamma+\sum_{n=0}^{\infty}\left(\frac{1}{n+1}-\frac{1}{n+z}\right), \quad z\neq0,-1,-2,\dots,
 \]
and \eqref{eq:m_fction_alpha=1}, one obtains
 \[
  \frac{1}{\pi}\lim_{\epsilon\to0+}\Im m(x+\mathrm{i}\epsilon,1)=\frac{2}{\pi}\sum_{n=0}^{\infty}(-1)^{n}\frac{2n+1}{(2n+1)^{2}+x^{2}}
  =\frac{1}{2\cosh\left(\pi x/2\right)},
 \]
 as expected. The last equality is the Mittag-Leffler expansion of the hyperbolic secant.
\end{rem}

\section{The non-self-adjoint case}\label{sec:nsa_case}

\subsection{The case $|\alpha|<1$}

First we extend the formula \eqref{eq:m_fction_0<alpha<1} for the Weyl $m$-function for $|\alpha|<1$.

\begin{prop}\label{prop:m-func_|alpha|<1}
 The formula \eqref{eq:m_fction_0<alpha<1} for the Weyl $m$-function of $J(\alpha)$ remains valid for $0<|\alpha|<1$
 and $z\in\rho(J(\alpha))$.
\end{prop}

\begin{proof}
 Let us temporarily denote the RHS of \eqref{eq:m_fction_0<alpha<1} by $\mathrm{ML}(z,\alpha)$.
 Recall first that $K$ and $q$ are analytic non-constant functions on the set $\mathbb{C}\setminus\left((-\infty,-1]\cup[1,\infty)\right)$,
 which is proved within the theory of elliptic functions, see \cite[Chp.~7,\S~8]{chandrasekharan85} and also \cite[Sec.~4]{walker_rslps03}.  Moreover, $|q|<1$ for all $\alpha\in\mathbb{C}\setminus\left((-\infty,-1]\cup[1,\infty)\right)$.
 Consequently, $\mathrm{ML}(z,\cdot)$ is a meromorphic function on $\mathbb{C}\setminus\left((-\infty,-1]\cup[1,\infty)\right)$.
 
 Take $r \in (0,1)$. By Lemma~\ref{lem:pert}.(ii), we known that there exists a non-empty open set $\mathcal{Z}_{r}\subset\C$ such that the function $\alpha\mapsto(J(\alpha)-z)^{-1}$ is analytic on $B_{r}(0)$ for all $z\in\mathcal{Z}_{r}$. Hence the same holds true for the $m$-function $m(z,\cdot)$. Since, for $z\in\mathcal{Z}_{r}$, the equality $m(z,\alpha)=\mathrm{ML}(z,\alpha)$, 
 i.e.~an equality between two meromorphic functions in $\alpha$, holds true for all $\alpha\in(0,1)$, it has to remain valid for all $\alpha\in B_r(0)$. Hence, $m(z,\alpha)=\mathrm{ML}(z,\alpha)$ for all $\alpha\in B_r(0)$ and all $z\in\mathcal{Z}_{r}$.
 
 At the same time, both functions $m(\cdot,\alpha)$ and $\mathrm{ML}(\cdot,\alpha)$  are analytic in the set 
 $\rho(J(\alpha))\cap\big(\mathbb{C}\setminus\frac{\pi}{2K(\alpha)}\left(2\mathbb{Z}+1\right)\big)$.
 So the equality $m(z,\alpha)=\mathrm{ML}(z,\alpha)$ remains true on this domain for $z$, by the analyticity argument in $z$.
 
 Notice that function $m(\cdot,\alpha)$ has singularities at points $z\in\frac{\pi}{2K(\alpha)}\left(2\mathbb{Z}+1\right)$. Thus, since the $m$-function is analytic on the
 resolvent set, we have $\frac{\pi}{2K(\alpha)}\left(2\mathbb{Z}+1\right)\subset\sigma(J(\alpha))$. All in all, we get the equality $m(z,\alpha)=\mathrm{ML}(z,\alpha)$ for all $\alpha\in B_r(0)$ and all $z\in\rho(J(\alpha))$.
 Since $r$ is an arbitrary number smaller than $1$, the last claim can be extended to all $\alpha$ with $|\alpha|<1$.
\end{proof}

 Recall that in the case of Jacobi operators with compact resolvent (non-self-adjoint in general), 
 the Weyl $m$-function is a meromorphic function and the set of its poles coincides with the spectrum of
 the Jacobi operator. In addition, the algebraic multiplicity of an eigenvalue is the same as the order
 of the pole, see \cite[Thm.~5.3, Cor.~5.5]{beckermann_jat99}.
 The following statement follows immediately from Proposition \ref{prop:m-func_|alpha|<1} and the formula 
 \eqref{eq:m_fction_0<alpha<1}. The claim holds true also for $\alpha=0$, as one readily verifies recalling that $K(0)=\pi/2$.

\begin{thm}\label{thm:spec_J}
 Let $|\alpha|<1$, then
  \[
  \sigma\left(J(\alpha)\right)=\sigma_{p}\left(J(\alpha)\right)=\{\lambda_N\}_{N \in \Z },
  \]
where 
  \begin{equation}\label{eq:EV.def}
  \lambda_N = \frac{\pi}{2K(\alpha)}(2N+1), \quad N \in \Z.
  \end{equation}
In addition, all $\lambda_N$ are simple, i.e. have the algebraic multiplicity equal to one. In particular, if $\Re\alpha=0$, then $\sigma\left(J(\alpha)\right) \subset \R$.
\end{thm}

\subsection{Eigenvectors and their asymptotics}

Define 
\begin{equation}\label{eq:C_k.D_k_defs}
\begin{aligned}
C_{k}=C_{k}(z,\alpha)&:=\int_{0}^{2K(\alpha)}e^{-zt}\cn(t,\alpha)\sn^{k}(t,\alpha)\, \dd t,
\\
 D_{k}=D_{k}(z,\alpha)&:=\int_{0}^{2K(\alpha)}e^{-zt}\dn(t,\alpha)\sn^{k}(t,\alpha)\, \dd t, \quad k\in\N_0.
\end{aligned}
\end{equation}
The integration is carried out through the line segment in $\C$ connecting points $0$ and $2K(\alpha)$.
Integrals in \eqref{eq:C_k.D_k_defs} are well defined for any 
$\alpha\in\C\setminus\left((-\infty,-1]\cup[1,\infty)\right)$ since functions 
$\sn(uK(\alpha),\alpha)$, $\cn(uK(\alpha),\alpha)$ and $\dn(uK(\alpha),\alpha)$ 
are analytic in $u\in\R$. For our purposes, it is sufficient to restrict $\alpha$
on the unit disk $|\alpha|\leq1$ and exclude the boundary points $\alpha=\pm1$.

Before we proceed with the further spectral analysis of $J(\alpha)$, we investigate the asymptotic behavior of $C_{k}$ and $D_{k}$ as $k\to\infty$. The integral form of the definition formulas \eqref{eq:C_k.D_k_defs} is suitable for the application of the saddle point method. However, some knowledge on values $|\sn(u K(\alpha),\alpha)|$ for $u\in[0,2]$ and $|\alpha| \leq1$ is necessary. The needed property is stated in the following lemma, proved in a separate paper \cite{ss15} devoted entirely to properties of function 
$\alpha\mapsto\sn(uK(\alpha),\alpha)$ for complex $\alpha$.

\begin{lem}[\cite{ss15}]\label{lem:min}
	Let $|\alpha|\leq1$ and $\alpha\neq\pm1$, then 
	\begin{equation}\label{eq:sn.max}
	\left|\sn\left(u K(\alpha),\alpha\right)\right| < 1, \quad \forall u \in [0,1).
	\end{equation}
Hence, taking into account that $\sn(K(\alpha),\alpha)=1$ and $\sn(2K(\alpha)-z,\alpha)=\sn(z,\alpha)$, the function $ u \mapsto \left|\sn\left(u K(\alpha),\alpha\right)\right|$, 
restricted to the interval $(0,2)$, has the unique global maximum at $u=1$ with the value equal to 1.
\end{lem}

\begin{prop}\label{prop:CD_asym}
	Let $|\alpha| \leq 1$, $\alpha \neq \pm 1$ and $z\in\C$, then
	\begin{equation}\label{eq:CD_asym}
	\begin{aligned}
	C_{k}(z,\alpha)&=\frac{\sqrt{2\pi}}{1-\alpha^{2}}ze^{-zK(\alpha)}\frac{1}{k^{3/2}}+O(k^{-5/2}), & k\to\infty,
\\
	D_{k}(z,\alpha)&=\sqrt{2\pi}e^{-zK(\alpha)}\frac{1}{k^{1/2}}+O(k^{-3/2}), & k\to\infty.
	\end{aligned}
	\end{equation}
\end{prop}
\begin{proof}
Notice that we have from \eqref{eq:sn.max} that if $z_1, z_2 \neq K(\alpha)$ lie on the line segments connecting $0$ with $K(\alpha)$ and $K(\alpha)$ with $2K(\alpha)$, respectively, then, for any bounded function $f$,
\begin{equation}\label{eq:exp.term}
\left|\left(\int_0^{z_1} + \int_{z_2}^{2K(\alpha)} \right) f(t) \sn^k(t,\alpha) \dd t \right| = O(p^k), \quad k \to \infty,
\end{equation} 
with $0<p<1$. On the other hand, on a sufficiently small neighborhood of $K(\alpha)$ where $\log \sn(t,\alpha)$ is analytic, we apply 
the saddle point method following \cite{olver97}. This yields $k^{-1/2}$ or $k^{-3/2}$ leading terms, thus the exponentially small term \eqref{eq:exp.term} can be neglected.
In the notation of \cite[Sec.~4.7, Thm.~7.1]{olver97}, we have
\begin{equation}
p(K(\alpha)) = 0, \quad p''(K(\alpha)) = 1-\alpha^2, \quad p'''(K(\alpha))=0
\end{equation}	
and, denoting by $q_C$, $q_D$ the corresponding functions $q$ for $C_k$, $D_k$, respectively,
\begin{align}
q_C(K(\alpha))&=0,& q_C''(K(\alpha))&=2ze^{-z K(\alpha)}\sqrt{1-\alpha^2}, 
\\
q_D(K(\alpha))&=e^{-z K(\alpha)} \sqrt{1-\alpha^2}, && 
\end{align}
thus we receive the asymptotic formulas \eqref{eq:CD_asym}.
\end{proof}

As showed in the following lemma, $C_k$ and $D_k$ satisfy certain difference equations and are closely related to the eigenvectors of $J(\alpha)$.

\begin{lem}\label{lem:CD_recur}
 Let $|\alpha|\leq1$, $\alpha\neq\pm1$, then, for all $z\in\C$, one has
\begin{equation}\label{eq:CD_eqs}
 \begin{aligned}
 -zD_{0}-\alpha^{2}C_{1}&=e^{-2K(\alpha)z}-1, &  kC_{k-1}-zD_{k}-\alpha^{2}(k+1)C_{k+1}&=0,&
 \\
 -zC_{0}-D_{1}&=-e^{-2K(\alpha)z}-1, & kD_{k-1}-zC_{k}-(k+1)D_{k+1}&=0, & k \in \N.
 \end{aligned}
\end{equation}
Moreover, the sequence $u=\{u_{n}\}_{n \in \N}$ defined by formulas
\begin{equation}\label{eq:def_u_k}
\begin{aligned}
	u_{2k+1}&=\ii(-1)^{k}\alpha^{k}e^{\ii K(\alpha) z}C_{2k}\left(\ii z,\alpha\right),
	\\
	u_{2k+2}&=(-1)^{k+1}\alpha^{k}e^{\ii K(\alpha) z}D_{2k+1}\left(\ii z,\alpha\right), \quad k \in \N_0,
\end{aligned}
\end{equation}
is the solution of the system of equations
\[
\begin{aligned}
u_{2}-zu_{1}&=-2\cos(K(\alpha)z),&
\\
(2k+1)u_{2k+2}-zu_{2k+1}+2k\alpha u_{2k}&=0,&\\
2k\alpha u_{2k+1}-zu_{2k}+(2k-1)u_{2k-1}&=0,&  k \in \N,
\end{aligned}
\]
or equivalently 
\begin{equation}
	\mathcal{J}(\alpha)u=zu-2\cos(K(\alpha)z)e_1, \quad z \in \C.
	\label{eq:Ju}
\end{equation}
\end{lem}

\begin{proof}
We obtain \eqref{eq:CD_eqs} by integrating \eqref{eq:C_k.D_k_defs} by parts, appealing to the derivative formulas \eqref{eq:dn.cn.ode} and using the identities \eqref{eq:dn.sn.cn.id} as well as the special values \eqref{eq:sn.cn.dn.sv}. 
The second statement follows immediately from equations~\eqref{eq:def_u_k} and \eqref{eq:CD_eqs}.
\end{proof}

Finally, we find the eigenvectors of $J(\alpha)$.

\begin{prop}\label{prop:eigenvec}
Let $0<|\alpha|<1$ and let $\{\lambda_N\}_{N \in \Z}$ be the simple eigenvalues of $J(\alpha)$, cf. \eqref{eq:EV.def}.
Then the eigenvectors of $J(\alpha)$ corresponding to the eigenvalues $\{\lambda_N\}_{N \in \Z}$ read
\begin{equation}\label{eq:eigenvec.def}
v_{2k+1}^{(N)}:=\ii(-1)^{k}\alpha^{k}C_{2k}\left(\ii \lambda_N,\alpha\right), \quad
v_{2k+2}^{(N)}:=(-1)^{k+1}\alpha^{k}D_{2k+1}\left(\ii \lambda_N,\alpha\right), \quad k \in \N_0,
\end{equation}
where $C_k$, $D_k$ are as in \eqref{eq:C_k.D_k_defs}. Moreover,
\begin{equation}\label{eq:eigenvec.asym}
\begin{aligned}
v_{2k+1}^{(N)}&=\ii\pi^{1/2}(-1)^{N+k}\frac{\lambda_N}{2(1-\alpha^2)}  \frac{\alpha^{k}}{k^{3/2}}+O(\alpha^{k}k^{-5/2}),& k\to\infty,
\\
v_{2k+2}^{(N)}&=\ii\pi^{1/2}(-1)^{N+k}\frac{\alpha^{k}}{k^{1/2}}+O(\alpha^{k}k^{-3/2}),& k\to\infty.
\end{aligned}
\end{equation}
\end{prop}

\begin{proof}
Since $ v^{(N)}=e^{-\ii K(\alpha) \lambda_N}u$, the second claim follows from Proposition \ref{prop:CD_asym}. Moreover, $0 \neq v^{(N)} \in \ell^2(\N)$. 
Since $\cos(K(\alpha)\lambda_N)=0$, equation \eqref{eq:Ju} yields
$   J(\alpha)v^{(N)}= \lambda_N v^{(N)}$.
\end{proof}

We conclude this subsection by showing the completeness of $\{v^{(N)}\}_{N\in\Z}$ in $\ell^{2}(\N)$.

\begin{prop}\label{prop:complete}
 Let $|\alpha|<1$, then the set of eigenvectors $\{v^{(N)}\}_{N \in \Z}$ of $J(\alpha)$, defined in \eqref{eq:eigenvec.def}, is complete in $\ell^{2}(\N)$.
\end{prop}
\begin{proof}
The proof is based on \cite[Cor.~XI.9.31]{dunford63} and the fact that $\{\lambda_N\}_{N \in \Z}$ are simple, 
see Theorem \ref{thm:spec_J}. 

It follows from \eqref{M.bound} and \eqref{Jres.repr} that there exist $c(\alpha), \delta(\alpha)>0$ such that
\begin{equation}
\|(J(\alpha)-z)^{-1}\| \leq C < \infty, \quad |z|>c(\alpha) \text{ and } ||\arg z| - \pi/2| < \delta(\alpha).
\end{equation}
Moreover, from Proposition~\ref{prop:sa_spec}, $(J(0)-z)^{-1} \in \mathcal S_p$ for every $p>1$, 
thus by \eqref{Jres.repr} and the ideal property of Schatten classes, we obtain that $(J(\alpha)-z)^{-1} \in \mathcal S_p$ for every $p>1$ as well. 
%
%
\end{proof}

\subsection{The case $|\alpha|=1$}

While $\sigma(J(\alpha))$ is discrete for $|\alpha|<1$, i.e., a set of isolated points in $\C$,
it suddenly fills the entire complex plane if $|\alpha|=1$, $\alpha \neq \pm 1$.

\begin{thm}
Let $|\alpha|=1$ and $\alpha \neq \pm 1$. Then
\begin{equation}
\sigma(J(\alpha)) = \sigma_{{\rm e}2}(J(\alpha))=\C.
\end{equation}		
\end{thm}

\begin{proof}
We proceed analogously to the proof of the second claim in Proposition \ref{prop:J_inv}, however, we begin
with the sequence $u$  defined in \eqref{eq:def_u_k}, use that $u$ satisfies \eqref{eq:Ju} and that we know 
the asymptotic behavior of $C_k$ and $D_k$, as stated in Proposition~\ref{prop:CD_asym}.
	
Take arbitrary $z \in \C$. We define a family of sequences $u(a)$, $a\in (0,1)$, by putting
\begin{equation}
u_{n}(a)=a^{n} u_{n}, \quad n \in \N,
\end{equation}	
where $u_n$ is as in \eqref{eq:def_u_k}. By Proposition \ref{prop:CD_asym}, one has
\begin{equation}\label{|uk|.asym}
|u_{2k+1}|= \frac{\sqrt \pi |z| }{2|1-\alpha^2|} \frac{1}{k^{3/2}}(1+o(1)),
\quad 
|u_{2k+2}|= \sqrt{\pi}\frac{1}{k^{1/2}}(1+o(1)), \quad k\to\infty.
\end{equation}
Thus $u(a) \in \ell^2(\N)$ for every $a\in (0,1)$ and there exist constants $C_1 >0$, $C_2 \geq 0$, independent of $a$, such that

\begin{equation}\label{ua.norm}
\|u(a)\|^2 \geq C_1 a^4 \left(\sum_{k=1}^\infty \frac{a^{4k}}{k} \right) -C_2 = -C_1 a^{4}\ln\left(1-a^4\right)-C_2.
\end{equation}	
Since $u$ satisfies \eqref{eq:Ju}, we get 
\begin{equation}
\begin{aligned}
\|(\mathcal J(\alpha)-z)u(a)\|^2  &= (1-a)^2 \sum_{k=1}^{\infty} a^{4k}|(2k+1)(a+1)u_{2k+2} - z u_{2k+1}|^2 
\\
& \quad + \frac{(1-a)^2}{a^2} \sum_{k=1}^{\infty} a^{4k}| 2k \alpha (a+1)u_{2k+1} - z u_{2k}|^2 
\\
& \quad + a^2 |(a-1) u_2 -2 \cos (K(\alpha) z)|^2.
\end{aligned}
\end{equation}
Hence, for $a \in (0,1)$, we have from \eqref{|uk|.asym} that
\begin{equation}\label{Jz.ub}
\begin{aligned}
\|(\mathcal J(\alpha)-z)u(a)\|^2 & \leq C_3 \frac{(1-a)^2}{a^2} \sum_{k=1}^{\infty} k a^{4k} + C_4 
= C_3 \frac{a^2}{(1+a)^{2}(1+a^2)^2} + C_4 \leq C_5,
\end{aligned}
\end{equation}
where $C_3$, $C_4$ and $C_5$ depend on $z$, but are independent of $a$. Notice that \eqref{Jz.ub} implies in particular that $u(a) \in \Dom(J(\alpha))$ for all $a \in (0,1)$.

By putting \eqref{|uk|.asym}, \eqref{ua.norm} and \eqref{Jz.ub} together, we obtain
\[
\forall k \in \N, \quad \lim_{a \to 1-} \frac{\langle e_k,u(a) \rangle}{\|u(a)\|} = 0 \quad \text{ and }
\quad 
\lim_{a \to 1-} \frac{\|(J(\alpha)-z)u(a)\|}{\|u(a)\|} = 0,
\]
thus $z \in \sigma_{e2}(J(\alpha))$ by \cite[Thm.~IX.1.3]{edmunds87}.
\end{proof}

\begin{rem}
The asymptotic formulas \eqref{eq:eigenvec.asym} have been derived as a direct consequence
of Proposition \ref{prop:CD_asym}, thus they remain valid also for $|\alpha|=1$, $\alpha\neq\pm1$.
For such $\alpha$, one observes that $v^{(N)}\notin\ell^{2}(\N)$. Consequently, since the solution of the difference equations $\mathcal J(\alpha) u = z u$ is unique up to a multiplicative constant, we get that 
$\{\lambda_N\}_{N \in \Z}$, defined as in \eqref{eq:EV.def}, are not eigenvalues of $J(\alpha)$.
\end{rem}

\subsection{Other properties of eigenvectors}

In this subsection, we provide some additional results related to eigenvectors of $J(\alpha)$ for $|\alpha|<1$.
Namely, we derive a generating function formula for eigenvectors, give a Rodriguez-like identity for orthogonal
polynomials associated with Jacobi matrix $\mathcal{J}(\alpha)$ and present an integral formula for the norm
of eigenprojections.

Note that formulas for $v^{(N)}$ given in Proposition \ref{prop:eigenvec} are expressible in terms of Fourier
coefficients of some analytic functions. Indeed, let us put
\begin{align}
 \mathcal{C}_{k}(s,\alpha)&:=e^{-\ii \frac s 2}\cn\left(\frac{K(\alpha)s}{\pi},\alpha\right)\sn^{2k}\left(\frac{K(\alpha)s}{\pi},\alpha\right), 
\label{eq:def_calC_k}
\\
\mathcal{D}_{k}(s,\alpha)&:=e^{-\ii \frac s 2} \dn \left(\frac{K(\alpha)s}{\pi},\alpha\right)\sn^{2k+1}\left(\frac{K(\alpha)s}{\pi},\alpha\right), \quad k \in \N_0,
\label{eq:def_calD_k}
\end{align}
and denote the corresponding Fourier coefficients by $\gamma_{n}(k)$ and $\delta_{n}(k)$, respectively, 
\begin{equation}
\gamma_{n}(k):=\frac{1}{2\pi}\int_{0}^{2\pi}e^{-\ii ns}\,\mathcal{C}_{k}(s,\alpha)\dd s,
\qquad 
\delta_{n}(k):=\frac{1}{2\pi}\int_{0}^{2\pi}e^{-\ii ns}\,\mathcal{D}_{k}(s,\alpha)\dd s.
\label{eq:def_gamma_delta}
\end{equation}
Then
\begin{equation}
 \mathcal{C}_{k}(s,\alpha)=\sum_{n\in\mathbb{Z}}\gamma_{n}(k)e^{\ii ns}, \qquad 
 \mathcal{D}_{k}(s,\alpha)=\sum_{n\in\mathbb{Z}}\delta_{n}(k)e^{\ii ns}
\label{eq:calC_k_CalD_k_Four_ser}
\end{equation}
and we have
\begin{equation}
 v^{(N)}_{2k+1}=2\ii K(\alpha) (-1)^{k}\alpha^{k}\gamma_{N}(k), \qquad  v^{(N)}_{2k+2}=2 K(\alpha) (-1)^{k+1}\alpha^{k}\delta_{N}(k), \quad k \in \N_0.
\label{eq:eigenvec_via_gamma_delta} 
\end{equation}

In \cite{carlitz_dmj60}, Carlitz investigated the sequence of polynomials $\{P_{n}\}_{n \in \N}$
defined recursively by the recurrence rule
\[
 P_{n+1}(x)=xP_{n}(x)-w_{n-1}^{2}P_{n-1}(x), \quad n\geq2,
\]
with initial conditions $P_{1}(x)=1$ and $P_{2}(x)=x$; sequence $\{w_{n}\}_{n\in\N}$ is as in \eqref{eq:def_seq_w}.
If we put
\begin{equation}
 p_{n}(x):=\left(\prod_{k=1}^{n-1}\frac{1}{w_{k}}\right)P_{n}(x)=\frac{1}{\alpha^{\lfloor(n-1)/2\rfloor}(n-1)!}P_{n}(x), \quad  n\in\N, \ x \in \C,
\label{eq:p_rel_P}
\end{equation}
then the sequence $p=\{p_{n}\}_{n \in \N}$ is the solution of the eigenvalue equation $\mathcal{J}(\alpha)p(x)=xp(x)$ normalized such that $p_{1}(x)=1$.
Since such a solution is uniquely determined by its first entry, the vector $v^{(N)}$ is a constant multiple of 
$p(\lambda_N)$ with $\lambda_N$ as in \eqref{eq:EV.def}. In detail, for $0<|\alpha|\leq1$, $\alpha\neq\pm1$, we have
\begin{equation}
  v^{(N)}_{k}=v^{(N)}_{1} p_{k} (\lambda_N), \quad k \in \N, \ N \in \Z.
\label{eq:eigenvec_normalization_1stto1}
\end{equation}

\begin{rem}
Notice that the Fourier expansion \eqref{eq:cn_Four_ser} can be used to evaluate $v^{(N)}_{1}$, 
\begin{equation} \label{eq:eigenvec_v_1}
v^{(N)}_{1}=\frac{\ii K(\alpha)}{\pi} \int_{0}^{2\pi}e^{-\ii\left(N+ \frac 12 \right)s}\cn\left(\frac{K(\alpha)s}{\pi},\alpha\right)\dd s = \frac{2\pi\ii}{\alpha} \frac{q^{N+\frac 12}}{1+q^{2N+1}}, \quad N\in\Z.
 \end{equation}
Taking into account that $P_{k}(x)$ is a monic polynomial in $x$ of degree $k-1$ and combining \eqref{eq:p_rel_P}, \eqref{eq:eigenvec_normalization_1stto1} and \eqref{eq:eigenvec_v_1}, we obtain
an asymptotic formula for $v_{k}^{(N)}$ as $N\to\infty$ with fixed $k\geq2$, 
\begin{equation}
 v_{k}^{(N)}=\frac{2\ii\pi^{k}}{\alpha^{\lfloor(k+1)/2\rfloor} K(\alpha)^{k-1}(k-1)!}N^{k-1}q^{N+1/2}+O\left(N^{k-2}q^{N}\right), \quad N\to\infty,
\label{eq:asymp_v_N_inf}
\end{equation}
where $0<|\alpha|\leq1$, $\alpha\neq\pm1$. 
It is a straightforward application of elementary properties of Jacobian elliptic functions to verify that
\[
C_{k}(-z,\alpha)=-e^{2K(\alpha)z}C_{k}(z,\alpha), \quad  \quad D_{k}(-z,\alpha)=e^{2K(\alpha)z}D_{k}(z,\alpha).
\]
Hence, taking into account that $\lambda_{-N-1}=-\lambda_{N}$, one deduces from \eqref{eq:eigenvec.def} that
\[
 v_{k}^{(-N-1)}=(-1)^{k+1}v_{k}^{(N)}, \quad k\in\N, \ N\in\Z.
\]
The last relation together with asymptotic formula \eqref{eq:asymp_v_N_inf} allows for obtaining the asymptotic formula for $v_{k}^{(N)}$ also as $N\to-\infty$.
\end{rem}

To supplement the knowledge about polynomials $p_{n}$, we provide a Rodriguez-like formula for $p_{n}$, 
which seems to be a new result.

\begin{prop}
For all $z\in\C$ and $k \in \N_0$, one has
 \begin{align}
p_{2k+1}(z)&=\frac{(-1)^{k}}{\alpha^{k}(2k)!}\,
\frac{\dd^{2k}}{\dd u^{2k}}\bigg|_{u=0}\!\!e^{\ii zu}\dn\left(u,\alpha\right)\left[\frac{u}{\sn\left(u,\alpha\right)}\right]^{2k+1},
\label{eq:Rodriguez_odd}
\\[2mm]
p_{2k+2}(z)&=\frac{\ii(-1)^{k+1}}{\alpha^{k}(2k+1)!}\,
\frac{\dd^{2k+1}}{\dd u^{2k+1}}\bigg|_{u=0}\!\!e^{\ii zu}\cn\left(u,\alpha\right)\left[\frac{u}{\sn\left(u,\alpha\right)}\right]^{2k+2}\!.
\label{eq:Rodriguez_even}
\end{align}
\end{prop}

\begin{proof}
We prove the statement for particular $z=\lambda_n$ with $n\in\Z$, see \eqref{eq:EV.def}.
Since both sides of equalities \eqref{eq:Rodriguez_odd} and \eqref{eq:Rodriguez_even} are polynomials in $z$, these
identities then hold for all $z\in\C$. 
 
By splitting the integral in \eqref{eq:def_gamma_delta} for $\gamma_{n}(k)$ to two integrals 
over $(0,\pi)$ and $(\pi,2\pi)$, applying the substitution $s=t-2\pi$ in the second one and using identities \eqref{eq:sn.cn.id.2K}, one finds
$$\gamma_{n}(k)
= \frac{1}{2\pi}\int_{-\pi}^{\pi}e^{\ii(n+1)s}\,\mathcal{C}_{k}(s,\alpha) \dd s.
$$
Take the parallelogram with vertices at points $\pm \pi$, $\pm \pi+2\pi\ii K'(\alpha)/K(\alpha)$.
Integrating the function
\[
z\mapsto e^{\ii(n+1)z}\mathcal{C}_{k}(z,\alpha)
\]
over the boundary of this parallelogram and taking into account that the integrand is a $2\pi$-periodic function,
hence the integrals over the lateral sides cancel each other, we obtain
\begin{equation}
 \frac{1}{2\pi}\left(\int_{-\pi}^{\pi}+\int_{\pi+2\pi\ii \frac{K'(\alpha)}{K(\alpha)}}^{-\pi+2\pi\ii \frac{K'(\alpha)}{K(\alpha)}}\right)e^{\ii(n+1)z}\mathcal{C}_{k}(z,\alpha)\dd z=
 \ii\Res\left(e^{\ii(n+1)z}\mathcal{C}_{k}(z,\alpha), z=\ii\pi \frac{K'(\alpha)}{K(\alpha)} \right)\!,
\label{eq:comp_gamma_n_k_inproof}
\end{equation}
for the function $\mathcal{C}_{k}(\cdot,\alpha)$ has the only singularity within the parallelogram located at $z=\ii\pi \frac{K'(\alpha)}{K(\alpha)}$.
Next, we parametrize the complex line segment in the second integral on the LHS of~\eqref{eq:comp_gamma_n_k_inproof} such that $z=t+2\pi\ii K'(\alpha)/K(\alpha)$ where $-\pi\leq t \leq\pi$.
Taking further into account that
\[
 \mathcal{C}_{k}\left(t+2\pi\ii \frac{K'(\alpha)}{K(\alpha)},\alpha\right)=-e^{\pi K'(\alpha)/K(\alpha)}\mathcal{C}_{k}(t,\alpha),
\]
as one deduces with the aid of \eqref{eq:sn.cn.id.2iK'}, the equation~\eqref{eq:comp_gamma_n_k_inproof} can be written as
\[
 (1+q^{2n+1})\gamma_{n}(k)=\ii\Res\left(e^{\ii(n+1)z}\mathcal{C}_{k}(z,\alpha), z=\ii\pi \frac{K'(\alpha)}{K(\alpha)}\right)\!,
\]
where we have substituted for the nome $q=\exp(-\pi K'(\alpha)/K(\alpha))$.
Note the singularity of the function $\mathcal{C}_{k}(\cdot,\alpha)$ at $z=\ii\pi K'(\alpha)/K(\alpha)$  is
a pole of order $2k+1$. Thus, using identities \eqref{eq:sn.cn.id.iK'} in the second step, one gets
\begin{align}
&\Res\left(e^{\ii(n+1)z}\mathcal{C}_{k}(z,\alpha), z=\ii\pi \frac{K'(\alpha)}{K(\alpha)}\right)\\
&\qquad \qquad =
\frac{1}{(2k)!}\frac{\dd^{2k}}{\dd z^{2k}}\bigg|_{z=\ii\pi \frac {K'(\alpha)}{K(\alpha)}} \left(z-\ii\pi\frac{K'(\alpha)}{K(\alpha)}\right)^{\!2k+1}\!\!e^{\ii(n+1)z}\mathcal{C}_{k}(z,\alpha)
\\
&\qquad \qquad =-\frac{\ii q^{n+\frac 12}}{\alpha^{2k+1}(2k)!}\frac{\dd^{2k}}{\dd z^{2k}}\bigg|_{z=0}e^{\ii(n+ \frac 12)z}\dn\left(\frac{K(\alpha)z}{\pi},\alpha\right)\left[\frac{z}{\sn\left(\frac{K(\alpha)z}{\pi},\alpha\right)}\right]^{2k+1}\!.
\end{align}
Consequently, we arrive at the formula
\[
 \gamma_{n}(k)=\frac{\pi}{K(\alpha)}\frac{q^{n+\frac 12}}{1+q^{2n+1}}\frac{1}{\alpha^{2k+1}(2k)!}\,
 \frac{\dd^{2k}}{\dd u^{2k}}\bigg|_{u=0}\!\!e^{\ii(n+ \frac 12)\pi \frac u{K(\alpha)} }\dn\left(u,\alpha\right)\left[\frac{u}{\sn\left(u,\alpha\right)}\right]^{2k+1}.
\]
Now, it suffices to apply identities \eqref{eq:eigenvec_via_gamma_delta}, \eqref{eq:eigenvec_normalization_1stto1} and 
\eqref{eq:eigenvec_v_1} to obtain \eqref{eq:Rodriguez_odd} with $z=\lambda_n$.

The second identity \eqref{eq:Rodriguez_even} is to be verified in a similar way. This time one deduces that 
\begin{align}
 \delta_{n}(k)& =-\frac{1}{2\pi}\int_{-\pi}^{\pi}e^{\ii(n+1)s}\,\mathcal{D}_{k}(s,\alpha)\dd s & 
\\
&=-\frac{\pi}{K(\alpha)}\frac{q^{n+ \frac 12}}{(1+q^{2n+1})}\frac{1}{\alpha^{2k+1}(2k+1)!}
\\& \quad \times \,
 \frac{\dd^{2k+1}}{\dd u^{2k+1}}\bigg|_{u=0}
\! \!   e^{\ii(n+\frac 12)\pi \frac{u}{K(\alpha)}}\cn\left(u,\alpha\right)
 \left[\frac{u}{\sn\left(u,\alpha\right)}\right]^{2k+2}\!.&  \qedhere
\end{align}
\end{proof}

Next, we derive some generating functions formulas for sequences $\gamma_{N}(k)$ and $\delta_{N}(k)$ with $N$ fixed.
They may be deduced from the result of Carlitz, see \cite[Eqs. (7.8), (7.9)]{carlitz_dmj60}, 
although the formulas there are treated rather as formal series, no comment on the convergence is given and 
$0<\alpha<1$ is assumed.

\begin{prop}\label{prop:gener_func_gamma_delta}
 If $0<|\alpha|<1$, then for $N\in\Z$ and $t$ from a neighborhood of the real line, one has
 \begin{align}
  \cn\left(\frac{K(\alpha)}{\pi}t,\alpha\right)\sum_{k=0}^{\infty}\gamma_{N}(k)\alpha^{2k}\sn^{2k}\left(\frac{K(\alpha)}{\pi}t,\alpha\right) &=
  \frac{\pi}{\alpha K(\alpha)}\frac{q^{N+ \frac 12}}{1+q^{2N+1}}\cos\left(\left(N+ \frac 12\right)t\right)\!, 
\\
  \dn\left(\frac{K(\alpha)}{\pi}t,\alpha\right)\sum_{k=0}^{\infty}\delta_{N}(k)\alpha^{2k}\sn^{2k+1}\left(\frac{K(\alpha)}{\pi}t,\alpha\right) &=
  \frac{\ii\pi}{\alpha K(\alpha)}\frac{q^{N+ \frac12}}{1+q^{2N+1}}\sin\left( \left(N+ \frac 12\right)t\right)\!.
 \end{align}
 
\end{prop}

\begin{proof}
We prove in detail the first formula.
By Lemma \ref{lem:min} and the analyticity of the function $t \mapsto \sn\left(\frac{K(\alpha)}{\pi}t,\alpha\right)$ on a neighborhood of $\R$, 
 there is an open set $U\subset\C$ such that $\R\subset U$ and
 \[
   \left|\sn\left(\frac{K(\alpha)}{\pi}t,\alpha\right)\right|<\frac{1}{\sqrt{|\alpha|}}, \quad t \in U.
 \]
Since, in addition, by \eqref{eq:eigenvec.asym} and $\eqref{eq:eigenvec_via_gamma_delta}$,
$\gamma_{N}(k)=O\left(k^{-3/2}\right)$ as $k\to\infty$. Hence, for $0<|\alpha|<1$, the series on the LHS of the first generating formula converges locally uniformly in $U$.
 
By using the definition \eqref{eq:def_gamma_delta} of $\gamma_{N}(k)$, interchanging the sum and integral and summing up, one arrives at
\begin{align}
&\cn\left(\frac{K(\alpha)}{\pi}t,\alpha\right)\sum_{k=0}^{\infty}\gamma_{N}(k)\alpha^{2k}\sn^{2k}\left(\frac{K(\alpha)}{\pi}t,\alpha\right) =
\\
& \qquad\int_{0}^{2\pi} \frac{e^{-\ii(N+ \frac 12)s}}{2 \pi} \frac{\cn\left(\frac{K(\alpha)}{\pi}t,\alpha\right)\cn\left(\frac{K(\alpha)}{\pi}s,\alpha\right)}{1-\alpha^{2}\sn^{2}\left(\frac{K(\alpha)}{\pi}t,\alpha\right)\sn^{2}\left(\frac{K(\alpha)}{\pi}s,\alpha\right)}\dd s. 
\end{align}
Applying the identity \eqref{eq:cn.add.1}, 
one evaluates the integral with the aid of Fourier expansion \eqref{eq:cn_Four_ser} of the function $\cn$, for
 \[
  \int_{0}^{2\pi}e^{-\ii(N+ \frac 12)s}\cn\left(\frac{K(\alpha)}{\pi}\left(t+s\right)\right)\dd s=
  \frac{2\pi^{2}}{\alpha K(\alpha)}\frac{q^{N+ \frac 12}}{1+q^{2N+1}}e^{\ii(N+ \frac12)t}.
 \]
 
The second generating function formula can be obtained similarly, one applies the identity \eqref{eq:cn.add.2} and proceeds analogously.
\end{proof}

\begin{cor}\label{cor:series_gg_dd}
 Let $M,N \in \Z$ and $0<|\alpha|<1$. Then it holds
 \begin{align}
  \sum_{k=0}^\infty \gamma_{N}(k)\gamma_{M}(k)\alpha^{2k}
  &=\frac{\pi}{2\alpha K(\alpha)}\frac{q^{N+ \frac 12}}{1+q^{2N+1}}\delta_{M,N}.
\\
  \sum_{k=0}^\infty \delta_{N}(k)\delta_{M}(k)\alpha^{2k}
  &=\frac{\pi}{2\alpha K(\alpha)}\frac{q^{N+ \frac 12}}{1+q^{2N+1}}\delta_{M,N}.
\end{align}
\end{cor}

\begin{proof}
 Multiply the first identity in Proposition \ref{prop:gener_func_gamma_delta} by 
 $(2\pi)^{-1}e^{-\ii(M+1/2)t}$ and integrate w.r.t. $t$ from $0$ to $2\pi$. The second 
 formula is to be derived analogously.
\end{proof}

It would be interesting to know whether the set of eigenvectors $\{v^{(N)}\}_{N\in\Z}$ 
forms a basis of $\ell^{2}(\N)$, or not. From this point of view, it is useful 
to have some knowledge on the norm of the eigenprojections 
$\{Q_{N}\}_{N \in \Z}$ corresponding to the eigenvalues $\{\lambda_N\}_{N \in \Z}$. 
Designating the dependence on $\alpha$ in the eigenvectors by writing $v^{(N)}=v^{(N)}(\alpha)$ and observing that $v^{(N)}(\overline{\alpha})$ is the eigenvector of $J^{*}(\alpha)$ corresponding to the eigenvalue
$\overline{\lambda_N}$, we have that 
\begin{equation}\label{PN.def}
Q_{N} =\frac{\left\langle v^{(N)}(\overline{\alpha}), \cdot \right\rangle}{\left\langle v^{(N)}(\overline{\alpha}), v^{(N)}(\alpha) \right\rangle}
v^{(N)}(\alpha).
\end{equation}
Since $\overline{v^{(N)}(\alpha)}=-v^{(N)}(\overline{\alpha})$, one obtains
\[
\|Q_{N}\|=\frac{\|v^{(N)}(\alpha)\|^{2}}{|\left\langle v^{(N)}(\overline{\alpha}), v^{(N)}(\alpha) \right\rangle|}.
\]
Corollary \ref{cor:series_gg_dd} enable us to derive an integral formula for $\|Q_N\|$, nevertheless, it does not give a complete answer on the behavior of $\|Q_N\|$ yet as finding an asymptotic formula for $\|v^{(N)}\|$ as $N \to \infty$ seems to be a not easy task.

\begin{prop}
Let $0<|\alpha|<1$, $N\in\Z$ and $Q_N$ be as in \eqref{PN.def}. Then one has
\begin{equation}
\|Q_N\| = \frac{|\alpha|}{4|K(\alpha)|\pi} \frac{|1+q^{2N+1}|}{|q|^{N+\frac 12}} \|v^{(N)}(\alpha)\|^{2}
\end{equation}
and
 \[
  \|v^{(N)}(\alpha)\|^{2}=\frac{|K(\alpha)|^{2}}{\pi^{2}}\int_{0}^{2\pi}\int_{0}^{2\pi}e^{-\ii(N+ \frac 12)(u+v)}\,
  \frac{c(u) \overline{c(v)} - s(u) \overline{s(v)}  d(u) \overline{d(v)}}
  {1-|\alpha|^{2}s^{2}(u)\overline{s^{2}(v)}} \,\dd u\dd v
 \]
 where we use abbreviations $s(u)=\sn\left(\frac{K(\alpha)}{\pi}u,\alpha\right)$, 
 $c(v)=\cn\left(\frac{K(\alpha)}{\pi}v,\alpha\right)$, etc.
\end{prop}

\begin{proof}

By Corollary \ref{cor:series_gg_dd} and formulas \eqref{eq:eigenvec_via_gamma_delta}, one computes
\[
\left\langle v^{(N)}(\overline{\alpha}), v^{(N)}(\alpha) \right\rangle=
\frac{4K(\alpha)\pi}{\alpha} \frac{q^{N+1/2}}{1+q^{2N+1}}.
\]
In the RHS of the equality 
 \[
  \|v^{(N)}(\alpha)\|^{2}=-\sum_{n=1}^{\infty}v_{n}^{(N)}(\alpha)v_{n}^{(N)}(\overline{\alpha}),
 \]
substitute by formulas \eqref{eq:eigenvec_via_gamma_delta}, interchange the summation and integrals 
and sum it up.
\end{proof}

We conclude by comparing numerics of pseudospectrum of $J(\alpha)$ for $\alpha =0.5$ and $\alpha= 0.5 \ii$, see Figure~\ref{fig:ps}. 
The plots are computed in Mathematica as the $\log$ of the norm of the inverse of $J(\alpha)-z$ truncated to $1000 \times 1000$ matrix. 
Although all eigenvalues are real in both cases, the pseudospectra have completely different character and they suggest that the eigenvectors 
of $J(\alpha)$ for non-real $\alpha$, $|\alpha|<1$, do not form a Riesz basis (as otherwise the $\varepsilon$-pseudospectrum should be contained
in a  $\kappa \varepsilon$-neighborhood of eigenvalues with some $\kappa>0$, see, e.g.,~\cite{krejcirik_jmp15}). 
More plots of pseudospectra of $J(\alpha)$ with various values of $\alpha$ can be found in Appendix~\ref{app:pseudo}.

\begin{figure}[htb!]
\includegraphics[width=0.75\textwidth]{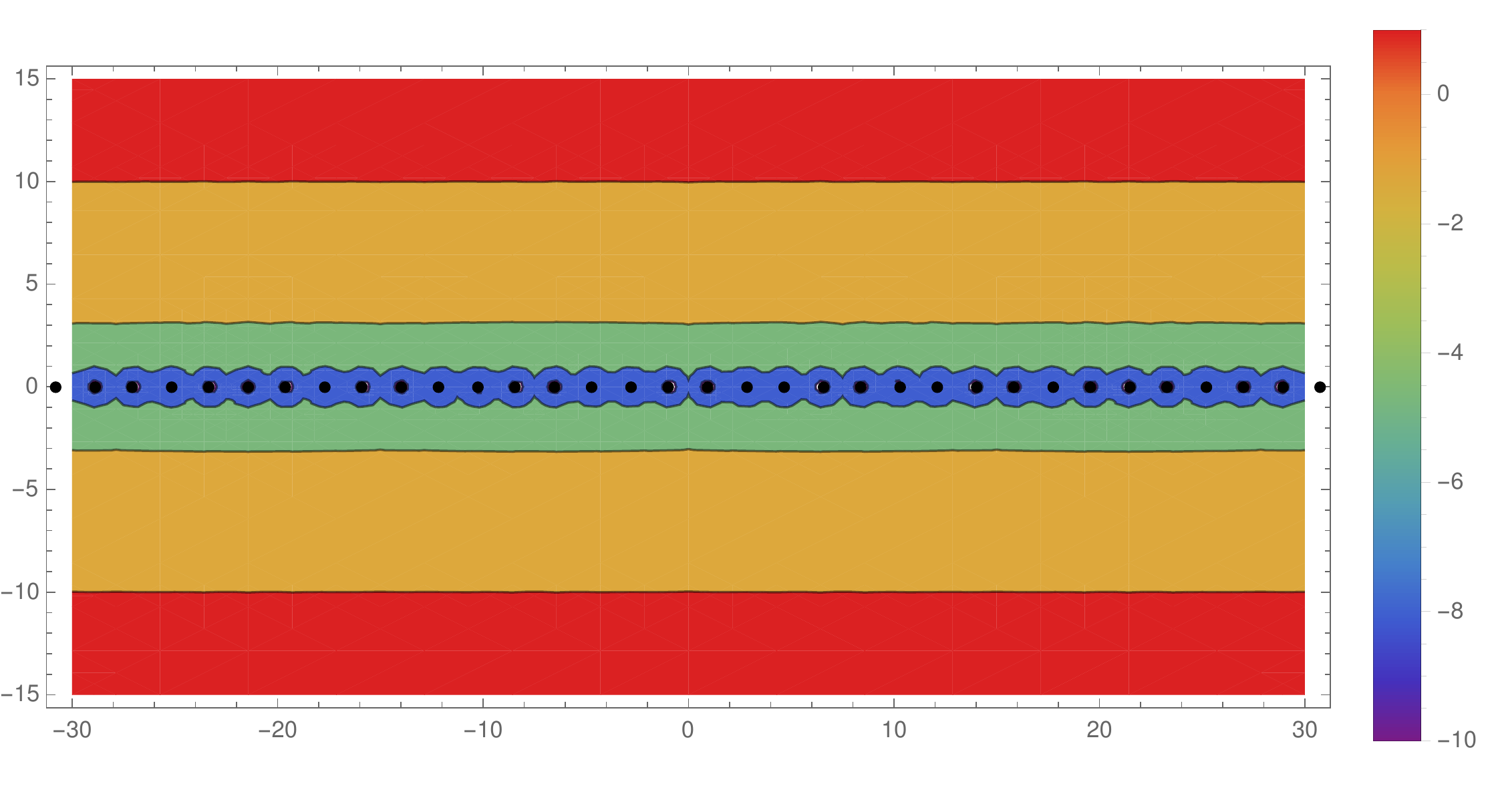} 
\includegraphics[width=0.75 \textwidth]{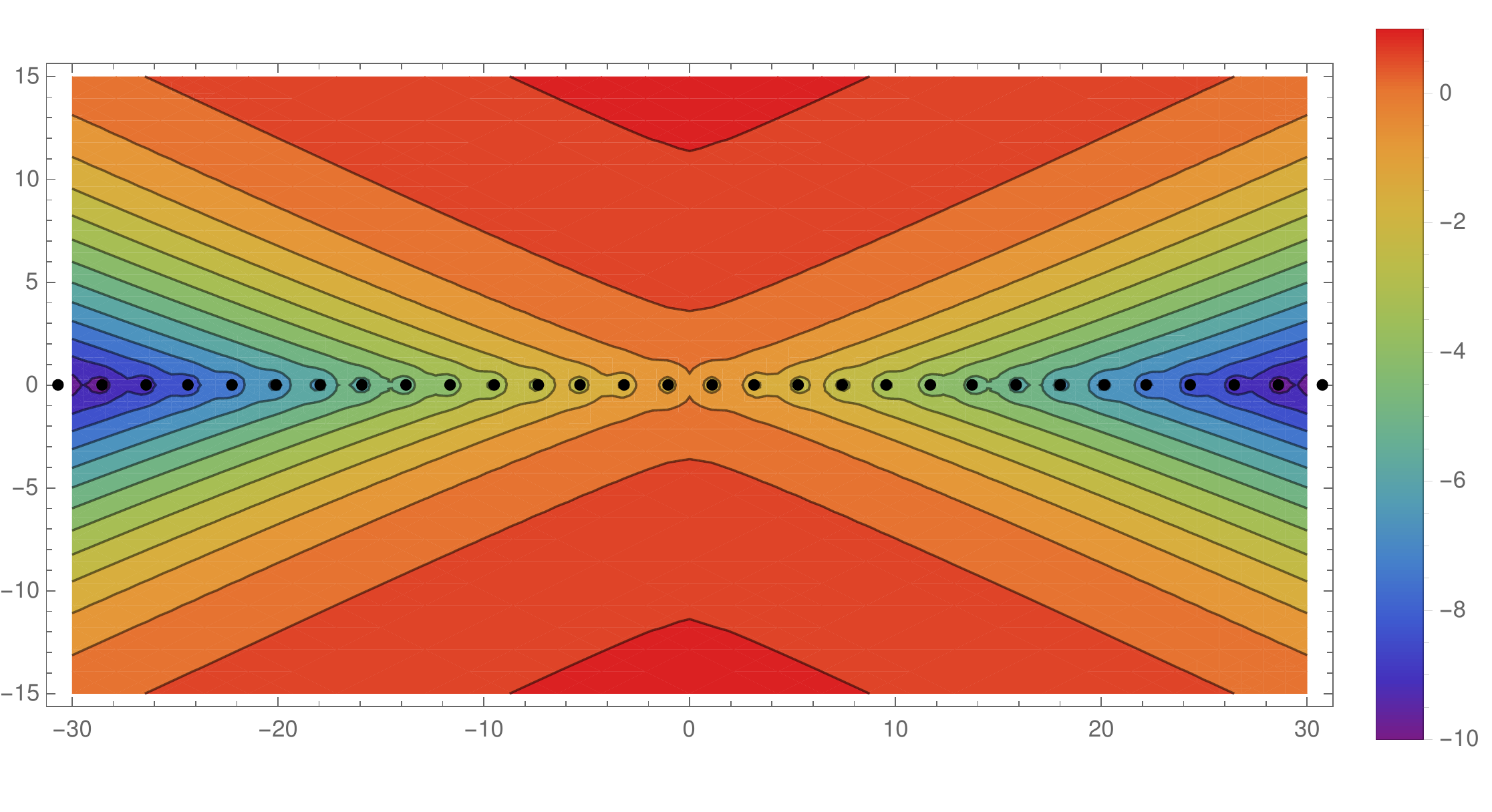}
\caption{Pseudospectra of $J(\alpha)$ for $\alpha=0.5$ (up) and $\alpha=0.5 \ii$ (down).}
\label{fig:ps}
\end{figure}

\section{The case $|\alpha|>1$}\label{sec:case_alp_geq_1}

For the sake of completeness, we describe the spectral properties of the Jacobi operator $J(\alpha)$ also in the case when $|\alpha|>1$.
However, the analysis is very similar to the case $|\alpha|<1$, therefore we provide only final formulas omitting detailed derivations.

In fact, the problem is reformulated using the operator $\tilde{J}(\beta)$ associated with Jacobi matrix 
\begin{equation}
\tilde{\mathcal{J}}(\beta):=\alpha^{-1}\mathcal{J}(\alpha), \quad \beta:=\alpha^{-1}.
\label{eq:def_tildeJ}
\end{equation}
Hence, if $|\alpha|>1$, then $0<|\beta|<1$ and operator $\tilde{J}(\beta)$ has discrete spectrum.

Taking \eqref{eq:def_tildeJ} together with \eqref{eq:powJ_eq_C} one obtains identities
\[
\langle e_{1}, \tilde{J}(\beta)^{2n+1}e_{1}\rangle=0, \quad \langle e_{1}, \tilde{J}(\beta)^{2n}e_{1}\rangle=\beta^{2n}C_{2n}\left(\beta^{-2}\right), \quad n\in\N_0, \ \beta \in \C \setminus\{0\}. 
\]
Thus, for $0< \beta <1$ and taking into account \eqref{eq:dn_tayl}, one derives the analogue of \eqref{eq:expizJ}, namely
\[
  \int_{\mathbb{R}}e^{ixz}\mathrm{d}\tilde{\mu}(x)=\dn(z,\beta),
\]
where $\tilde{\mu}(\cdot):=\langle e_{1}, E_{\tilde{J}}(\cdot)e_{1}\rangle$ where $E_{\tilde{J}}$ stands for the spectral 
measure of $\tilde{J}(\beta)$. The application of the inverse Fourier transform and the formula \eqref{eq:dn_Four_ser} then yields
\[
 \tilde{\mu}(t)=\frac{\pi}{2 K(\beta)}\delta(t)+\frac{\pi}{K(\beta)}\sum_{n=1}^{\infty}\frac{q^{n}}{1+q^{2n}}\left[\delta\left(t-\frac{n\pi }{K(\beta)}\right)+\delta\left(t+\frac{n\pi }{K(\beta)}\right)\right]\!.
\]
Consequently, one has $\sigma(\tilde{J}(\beta))=\frac{\pi}{K(\beta)}\Z$ for $0<\beta<1$.

For the corresponding Weyl $m$-function, one derives the Mittag-Leffler expansion
\[
 \tilde{m}(z,\beta):=\langle e_1,(\tilde{J}(\beta)-z)^{-1}e_1\rangle=
-\frac{\pi}{K(\beta)}\sum_{n=-\infty}^{\infty}\frac{q^{n}}{1+q^{2n}}\frac{1}{z-\frac{n\pi}{K(\beta)}}.
\]
which holds true for any $0<|\beta|<1$ and $z\in\rho(\tilde{J}(\beta))$. Consequently,
\[
\sigma(\tilde{J}(\beta))=\frac{\pi}{K(\beta)}\Z, \quad |\beta|<1,
\]
and all eigenvalues are simple.

From equations \eqref{eq:CD_eqs}, it follows that the vector $\tilde{u}=\{\tilde{u}_n\}_{n\in\N}$ defined by formulas
\[
  \tilde{u}_{2k+1}:=\ii(-1)^{k}\beta^{k}e^{\ii K(\beta) z}D_{2k}\left(\ii z,\beta\right),
	\quad
  \tilde{u}_{2k+2}:=(-1)^{k+1}\beta^{k+1}e^{\ii K(\beta) z}C_{2k+1}\left(\ii z,\beta\right), \quad k \in \N_0,
\]
satisfies
\[
 \mathcal{J}(\beta)\tilde{u}=z\tilde{u}-2\ii\sin(K(\beta)z)e_1, \quad z\in\C.
\]
Consequently, vectors $\tilde{v}^{(N)}$, $N\in\Z$, with entries
\[
  \tilde{v}_{2k+1}^{(N)}:=\ii(-1)^{k}\beta^{k}D_{2k}\left(\ii \tilde{\lambda}_{N},\beta\right),
	\quad
  \tilde{v}_{2k+2}^{(N)}:=(-1)^{k+1}\beta^{k+1}C_{2k+1}\left(\ii \tilde{\lambda}_{N},\beta\right), \quad k \in \N_0,
\]
are eigenvectors of $\tilde{J}(\beta)$ corresponding to eigenvalues 
$$\tilde{\lambda}_{N}:=\frac{\pi}{K(\beta)}N, \quad N \in \Z.$$ 
A straightforward
application of formulas \eqref{eq:CD_asym} yields the asymptotic relations
\begin{align*}
\tilde{v}_{2k+1}^{(N)}&=\ii\pi^{1/2}(-1)^{N+k}\frac{\beta^{k}}{k^{1/2}}+O(\beta^{k}k^{-3/2}),& k\to\infty,
\\
\tilde{v}_{2k+2}^{(N)}&=\ii\pi^{1/2}(-1)^{N+k+1}\frac{\tilde{\lambda}_{N}}{2(1-\beta^{2})}\frac{\beta^{k+1}}{k^{3/2}}+O(\beta^{k}k^{-3/2}),& k\to\infty.
\end{align*}

Orthogonal polynomials studied by Carlitz in \cite[\S\S~8]{carlitz_dmj60} are determined recursively by the recurrence rule
\[
 \tilde{P}_{n+1}(x)=x\tilde{P}_{n}(x)-\beta^{2}w_{n-1}^{2}\tilde{P}_{n-1}(x), \quad n\geq2,
\]
with initial conditions $\tilde{P}_{1}(x)=1$ and $\tilde{P}_{2}(x)=x$; sequence $\{w_{n}\}_{n\in\N}$ is as in \eqref{eq:def_seq_w} and $\alpha=\beta^{-1}$.
For the sequence of polynomials $\tilde{p}=\{\tilde{p}_{n}\}_{n \in \N}$ satisfying the eigenvalue equation $\tilde{\mathcal{J}}(\beta)\tilde{p}(x)=x\tilde{p}(x)$ 
and normalized such that $\tilde{p}_{1}(x)=1$, one gets
\[
 \tilde{p}_{n}(x)=\frac{1}{\beta^{\lfloor n/2\rfloor}(n-1)!}\tilde{P}_{n}(x), \quad  n\in\N.
\]
The eigenvectors $\tilde v^{(N)}$ are related to these polynomials by relation
\[
 \tilde{v}_{k}^{(N)}=2\pi\ii\frac{q^{N}}{1+q^{2N}}\ \tilde{p}_{k}(\tilde{\lambda}_{N}), \quad  k\in\N,\ N\in\Z.
\]

For $n\in\N$, one readily verifies that
\[
 \tilde{P}_{n}(z)=\beta^{n-1}P_{n}\left(\beta^{-1}z\right) \quad \mbox{ and } \quad \tilde{p}_{n}(z)=p_{n}\left(\beta^{-1}z\right).
\]
Consequently, identities \eqref{eq:Rodriguez_odd} and \eqref{eq:Rodriguez_even} yield
\begin{align*}
\tilde{p}_{2k+1}(z)&=\frac{(-1)^{k}\beta^{k}}{(2k)!}\,
\frac{\dd^{2k}}{\dd u^{2k}}\bigg|_{u=0}\!\!e^{\ii \frac{zu}{\beta}}\dn\left(u,\beta\right)\left[\frac{u}{\sn\left(u,\beta\right)}\right]^{2k+1}\!,
\\[2mm]
\tilde{p}_{2k+2}(z)&=\frac{\ii(-1)^{k+1}\beta^{k}}{(2k+1)!}\,
\frac{\dd^{2k+1}}{\dd u^{2k+1}}\bigg|_{u=0}\!\!e^{\ii \frac{zu}{\beta}}\cn\left(u,\beta\right)\left[\frac{u}{\sn\left(u,\beta\right)}\right]^{2k+2}\!, \quad k\in\N_0.
\end{align*}

Let us end with the integral formula for the norm of the eigenprojection
\[
 \tilde{Q}_{N} =\frac{\left\langle \tilde{v}^{(N)}(\overline{\beta}), \cdot \right\rangle}{\left\langle \tilde{v}^{(N)}(\overline{\beta}), \tilde{v}^{(N)}(\beta) \right\rangle}\tilde{v}^{(N)}(\beta).
\]
which reads
\[
 \|\tilde{Q}_{N}\| = \frac{1}{4|K(\beta)|\pi} \frac{|1+q^{2N}|}{|q|^{N}} \|\tilde{v}^{(N)}(\beta)\|^{2}
\]
with
 \[
  \|\tilde{v}^{(N)}(\beta)\|^{2}=
  \frac{|K(\beta)|^{2}}{\pi^{2}}\int_{0}^{2\pi}\int_{0}^{2\pi}e^{-\ii N(u+v)}\,\frac{ d(u) \overline{d(v)} - s(u)\overline{s(v)}c(u)\overline{c(v)}}
  {1-|\beta|^{2}s^{2}(u)\overline{s^{2}(v)}} \,\dd u\dd v
 \]
where we use abbreviations $s(u)=\sn\left(\frac{K(\beta)}{\pi}u,\beta\right)$, $c(v)=\cn\left(\frac{K(\beta)}{\pi}v,\beta\right)$, etc.

\section*{Acknowledgments}
The research of P.~S. is supported by the \emph{Swiss National Science Foundation}, SNF Ambizione grant No.\ PZ00P2\_154786.
F.~{\v S}. gratefully acknowledges the kind hospitality of the Mathematisches Institut at 
Universit{\"a}t Bern and in particular of Professor Christiane Tretter; his research was also supported by grant No. GA13-11058S of the Czech Science Foundation. 
We also thank a referee for the very stimulating report, in particular containing plots of pseudospectra of $J(\alpha)$ that we re-computed in Mathematica and included as Appendix~\ref{app:pseudo}.

\appendix
\section{Jacobian elliptic functions}
\label{app:elliptic}

Jacobian elliptic functions are deeply investigated and very well-known. For convenience, 
some of their selected properties, which are used within the paper, are summarized here.
As a primarily source we use \cite{lawden89}, other useful references are 
\cite[Chp.~16]{abramowitz64}, \cite{dlmf22} and \cite{akhiezer90}.

The (copolar) triplet of Jacobian elliptic functions $\sn(u,\alpha)$, $\cn(u,\alpha)$ and $\dn(u,\alpha)$ can be 
defined with the aid of Jacobi's theta functions, see \cite[Eqs.~(2.1.1-3)]{lawden89} (modulus $\alpha$ coincides
with $k$ in the Lawden's notation). Each of these functions is meromorphic in $u$ (for fixed $\alpha$) with simple 
poles and simple zeros and is meromorphic in $\alpha$ (for fixed $u$).
In most applications, the range for the modulus $\alpha$ is restricted to $0<\alpha<1$.
As such, all three functions are real-valued for $u\in\R$.

Taylor series expansions of Jacobian elliptic functions can be written in the form:
\begin{align}
 \sn(u,\alpha)&=\sum_{n=0}^{\infty}(-1)^{n}C_{2n+1}\left(\alpha^{2}\right)\frac{u^{2n+1}}{(2n+1)!},\label{eq:sn_tayl}\\
 \cn(u,\alpha)&=\sum_{n=0}^{\infty}(-1)^{n}C_{2n}\left(\alpha^{2}\right)\frac{u^{2n}}{(2n)!},\label{eq:cn_tayl}\\
 \dn(u,\alpha)&=\sum_{n=0}^{\infty}(-1)^{n}\alpha^{2n}C_{2n}\left(\alpha^{-2}\right)\frac{u^{2n}}{(2n)!}. \label{eq:dn_tayl}
\end{align}
Expansions \eqref{eq:sn_tayl}, \eqref{eq:cn_tayl} and \eqref{eq:dn_tayl} are absolutely convergent for $|\alpha|\leq1$
and $|u|<\pi/2$, see \cite[Thm.~3.2]{walker_rslps03}. 
For $n \in \N$, $C_{n}(x)$ is a polynomial in $x$ of degree $\lfloor (n-1)/2 \rfloor$ with positive integer
coefficients. No explicit formula for polynomials $C_{n}$ is known, although a lot of authors studied them and found various combinatorial interpretations for their coefficients. Let us mention at least 
\cite{flajoletfrancon_ejc89, viennot_jcts80}. Polynomials $C_{n}$ may be computed recursively by formulas
\begin{align}
 C_{2n+1}(x)&=\sum_{j+k=n}\binom{2n}{2j}C_{2j}(x)x^{k}C_{2k}\left(x^{-1}\right),
\\
 C_{2n+2}(x)&=\sum_{j+k=n}\binom{2n+1}{2j+1}C_{2j+1}(x)x^{k}C_{2k}\left(x^{-1}\right), \qquad n \in \N_0,
\end{align}
and $C_{0}(x)=1$. First few polynomials $C_{n}(x)$ read
\begin{align}
  C_{1}(x)&=1, \; C_{3}(x)=1+x, \; C_{5}(x)=1+14x+x^{2}, \; C_{7}(x)=1+135x+135x^{2}+x^{3},
\\
  C_{2}(x)&=1, \; C_{4}(x)=1+4x, \; C_{6}(x)=1+44x+16x^{2}, \; C_{8}(x)=1+408x+912x^{2}+64x^{3}.
\end{align}

Zeros, poles as well as periodicity properties of Jacobian elliptic functions are expressible in terms
of the complete elliptic integral of the first kind
\begin{equation}
 K(\alpha)=\int_{0}^{1}\frac{\mbox{d}t}{\sqrt{(1-t^{2})(1-\alpha^{2}t^{2})}},
\label{eq:def_K}
\end{equation}
where the principle square root is used. As function of $\alpha^{2}$, $K$ is analytic in $\C\setminus[1,\infty)$. 
Note that $K(\alpha)>0$ whenever $\alpha^{2}<1$. The conjugate elliptic integral $K'$ is defined as $K'(\alpha)=K(\alpha')$
where the complementary modulus $\alpha'$ satisfies $\alpha^{2}+\alpha'^{2}=1$. Similarly, as function of $\alpha^{2}$, 
$K'$ is analytic in $\C\setminus(-\infty,0]$. Finally, recall the nome $q(\alpha)=\exp(-\pi K'(\alpha)/K(\alpha))$; the dependence on the modulus $\alpha$ is suppressed in the notation for $q$.
Note that all functions $\sn$, $\cn$, $\dn$, $K$, $K'$ and $q$ are functions of $\alpha^{2}$ rather then $\alpha$.

Fourier series for Jacobian elliptic functions read 
\begin{align}
\sn(u,\alpha)&=\frac{2\pi}{\alpha K(\alpha)}\sum_{n=0}^{\infty}\frac{q^{n+1/2}}{1-q^{2n+1}}\sin\frac{(2n+1)\pi u}{2K(\alpha)},\label{eq:sn_Four_ser}\\
\cn(u,\alpha)&=\frac{2\pi}{\alpha K(\alpha)}\sum_{n=0}^{\infty}\frac{q^{n+1/2}}{1+q^{2n+1}}\cos\frac{(2n+1)\pi u}{2K(\alpha)},\label{eq:cn_Four_ser}\\
\dn(u,\alpha)&=\frac{\pi}{2K(\alpha)}+\frac{2\pi}{K(\alpha)}\sum_{n=1}^{\infty}\frac{q^{n}}{1+q^{2n}}\cos\frac{n\pi u}{K(\alpha)},\label{eq:dn_Four_ser}
\end{align}
where $|\Im(u/K(\alpha))|<\Im(\ii K'(\alpha)/K(\alpha))$.

Finally, we recall some special values, see \cite[Sec.~16.5--16.8]{abramowitz64},
\begin{equation}\label{eq:sn.cn.dn.sv}
\begin{aligned}
\sn(0,\alpha)&=\sn(2K(\alpha),\alpha)=0, \\
\cn(0,\alpha)&=\dn(0,\alpha)=\dn(2K(\alpha),\alpha)=-\cn(2K(\alpha),\alpha)=1,
\end{aligned}
\end{equation}
identities, see \cite[Sec.~16.9]{abramowitz64} and \cite[Sec.~16.8]{abramowitz64},
\begin{equation}\label{eq:dn.sn.cn.id}
\dn^{2}(z,\alpha)+\alpha^{2}\sn^{2}(z,\alpha)=1, \quad 
\sn^{2}(z,\alpha)+\cn^{2}(z,\alpha)=1,
\end{equation}
and
\begin{align}
\sn(u+2K(\alpha),\alpha)&=-\sn(u,\alpha),& \cn(u+2 K(\alpha),\alpha)&=-\cn(u,\alpha),
\label{eq:sn.cn.id.2K}
\\
\sn(u+2\ii K'(\alpha),\alpha)&=\sn(u,\alpha),& \cn(u+2\ii K'(\alpha),\alpha)&=-\cn(u,\alpha),
\label{eq:sn.cn.id.2iK'}
\\
\sn(u+\ii K'(\alpha),\alpha)&=\frac{1}{\alpha\sn(u,\alpha)}, & \cn(u+\ii K'(\alpha),\alpha)=&-\frac{\ii\dn(u,\alpha)}{\alpha\sn(u,\alpha)},
\label{eq:sn.cn.id.iK'}
\end{align}
addition formulas, see \cite[Eqs.~2.4.12, 2.4.14]{lawden89},
\begin{align}
\cn(u+v)+\cn(u-v)&=\frac{2\cn(u)\cn(v)}{1-\alpha^{2}\sn^{2}(u)\sn^{2}(v)}, \label{eq:cn.add.1}\\
\cn(u+v)-\cn(u-v)&=\frac{2\sn(u)\sn(v)\dn(u)\dn(v)}{1-\alpha^{2}\sn^{2}(u)\sn^{2}(v)}, \label{eq:cn.add.2}
\end{align}
and formulas for derivatives, see \cite[Sec.~16.16]{abramowitz64},
\begin{equation}\label{eq:dn.cn.ode}
\begin{aligned}
\frac{\partial}{\partial z}\sn(z,\alpha)&=\cn(z,\alpha)\dn(z,\alpha), 
&\frac{\partial}{\partial z}\cn(z,\alpha)=-\sn(z,\alpha)\dn(z,\alpha),
\\
\frac{\partial}{\partial z}\dn(z,\alpha)&=-\alpha^{2}\sn(z,\alpha)\cn(z,\alpha).&
\end{aligned}
\end{equation}

\section{Pseudospectra of $J(\alpha)$}
\label{app:pseudo}

We investigate numerically the pseudospectra of $J(\alpha)$ for $\alpha$ lying close to the unit circle, see Figure~\ref{fig:ps_circle}, 
and approaching $\ii$ from inside, see~Figure~\ref{fig:ps_im}. The plots suggest that in spite of the reality of the spectrum in some cases, the pseudospectra (and so the basis properties of eigenvectors) crucially depend on the self-adjointness of $J(\alpha)$.

\begin{figure}[htb!]
\includegraphics[width=0.36 \textwidth]{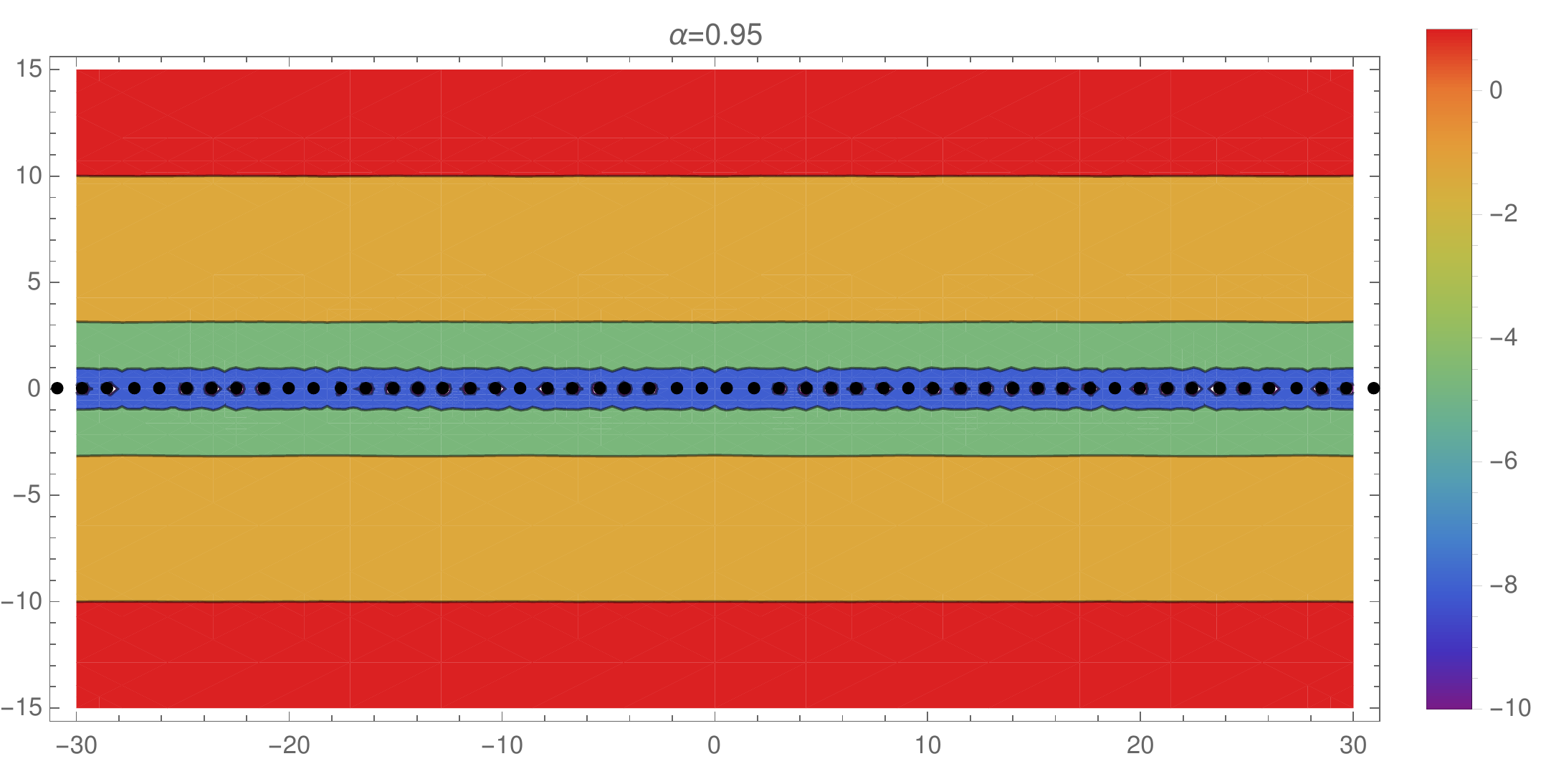}
\qquad
\includegraphics[width=0.36 \textwidth]{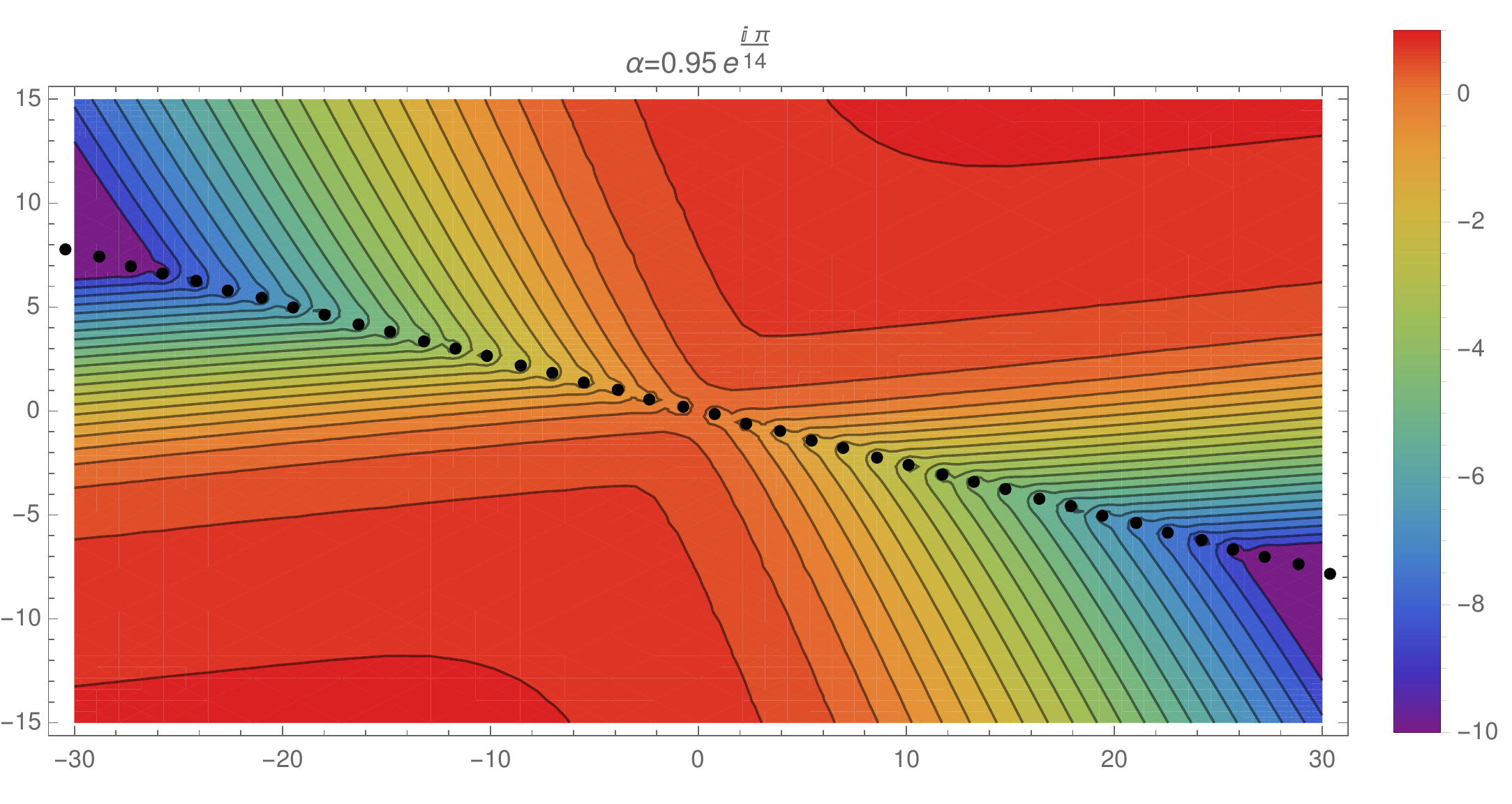}
\\
\includegraphics[width=0.36 \textwidth]{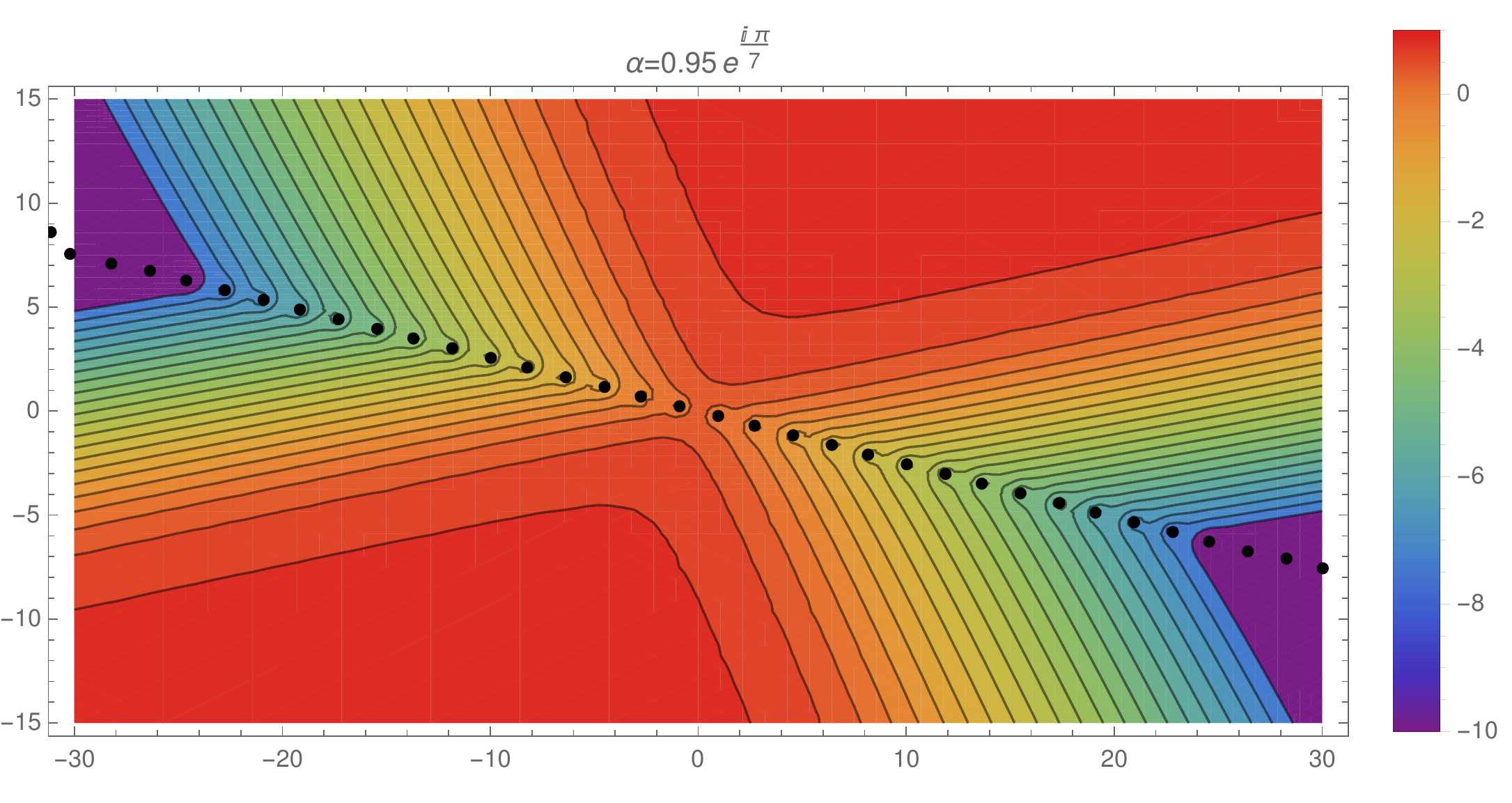}
\qquad
\includegraphics[width=0.36 \textwidth]{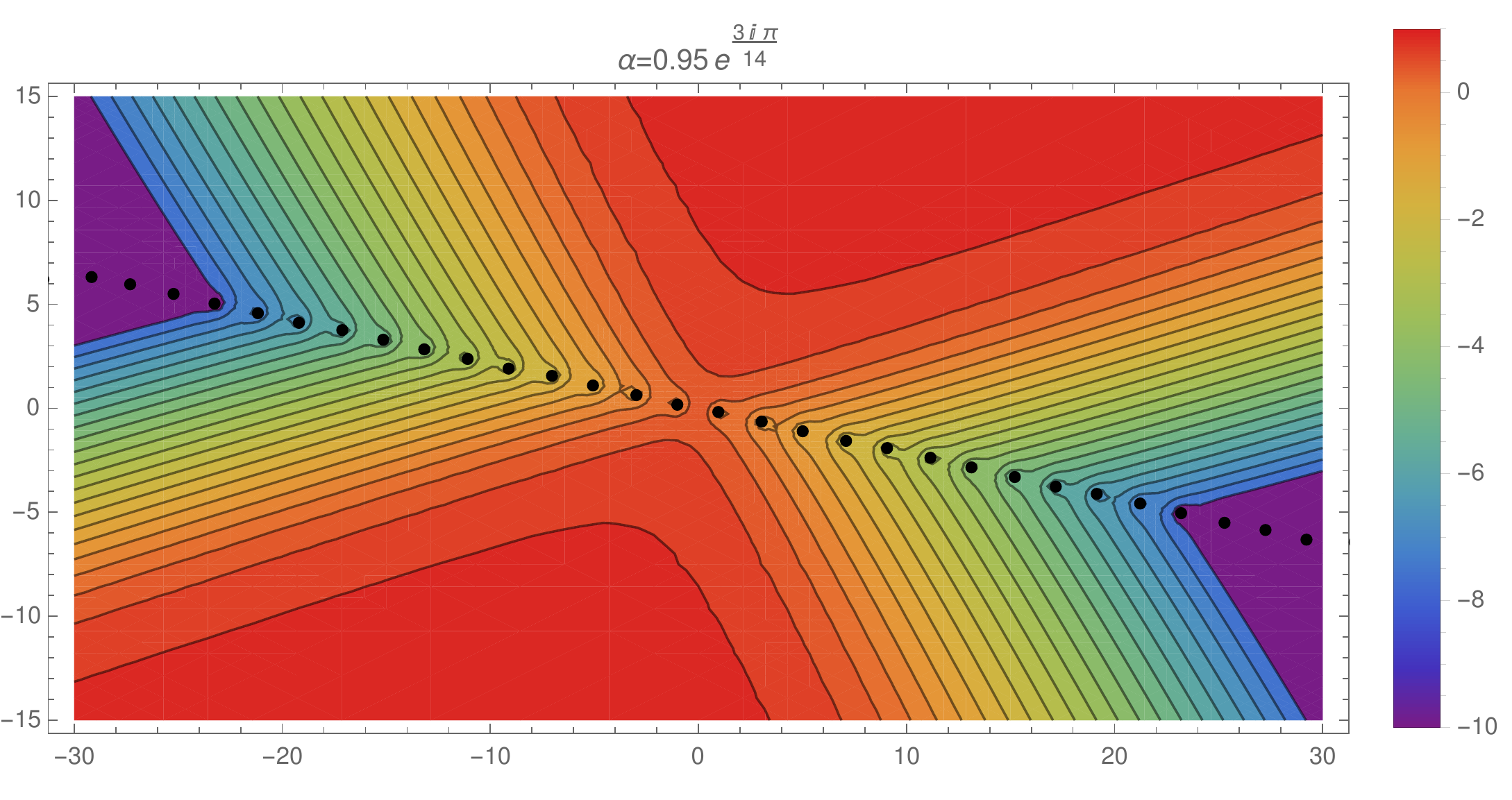}
\\
\includegraphics[width=0.36 \textwidth]{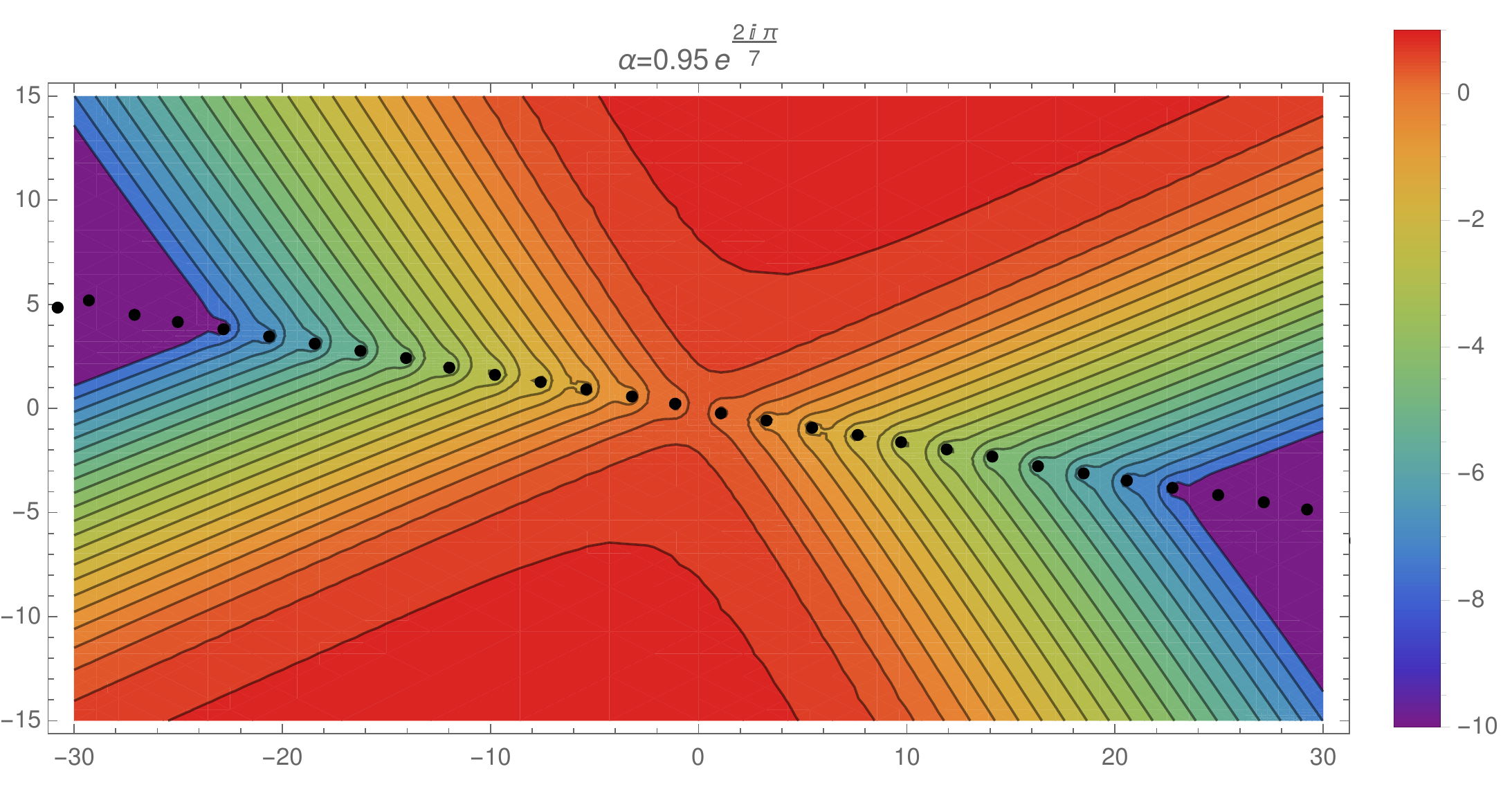}
\qquad
\includegraphics[width=0.36 \textwidth]{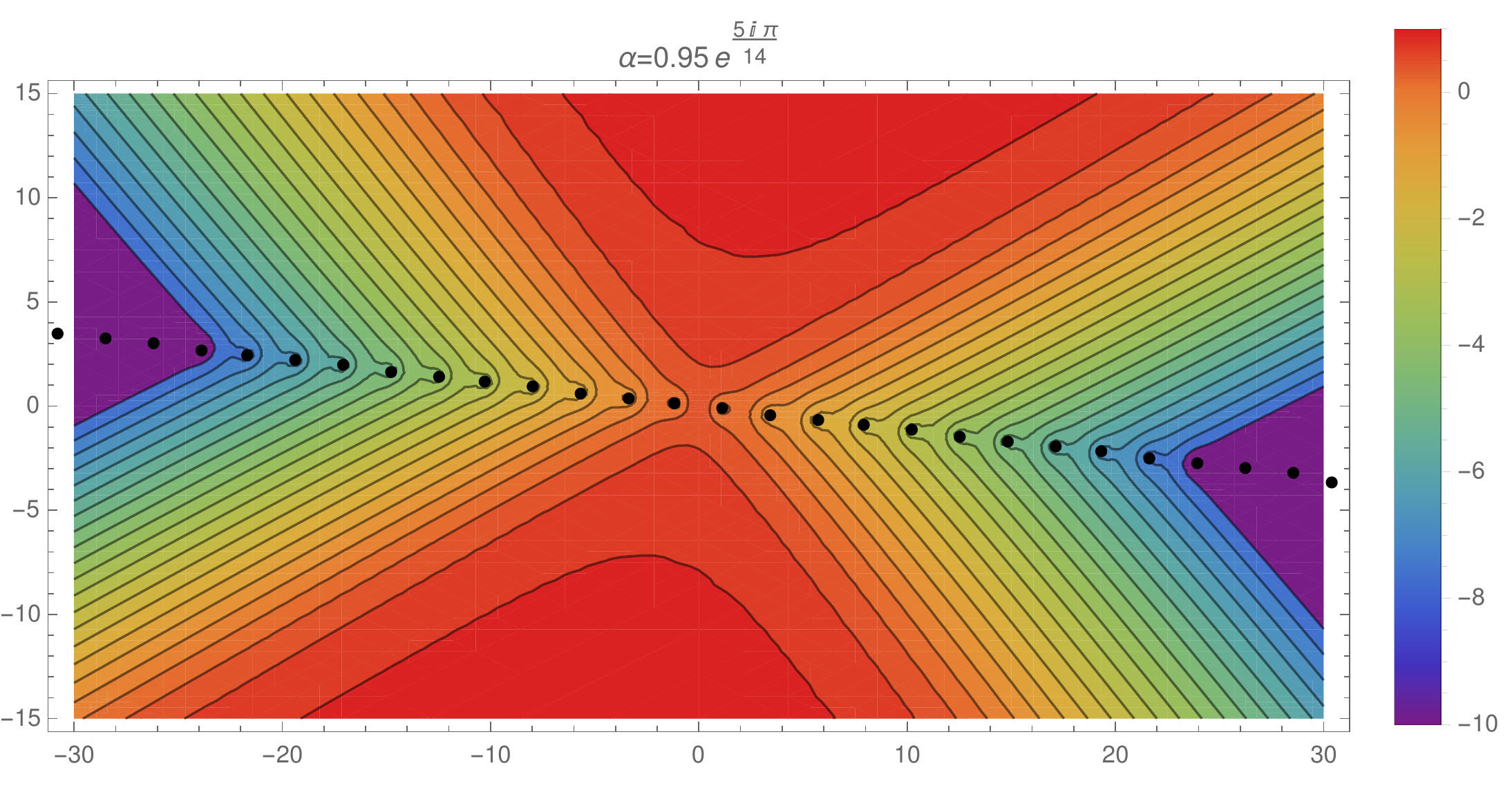}
\\
\includegraphics[width=0.36 \textwidth]{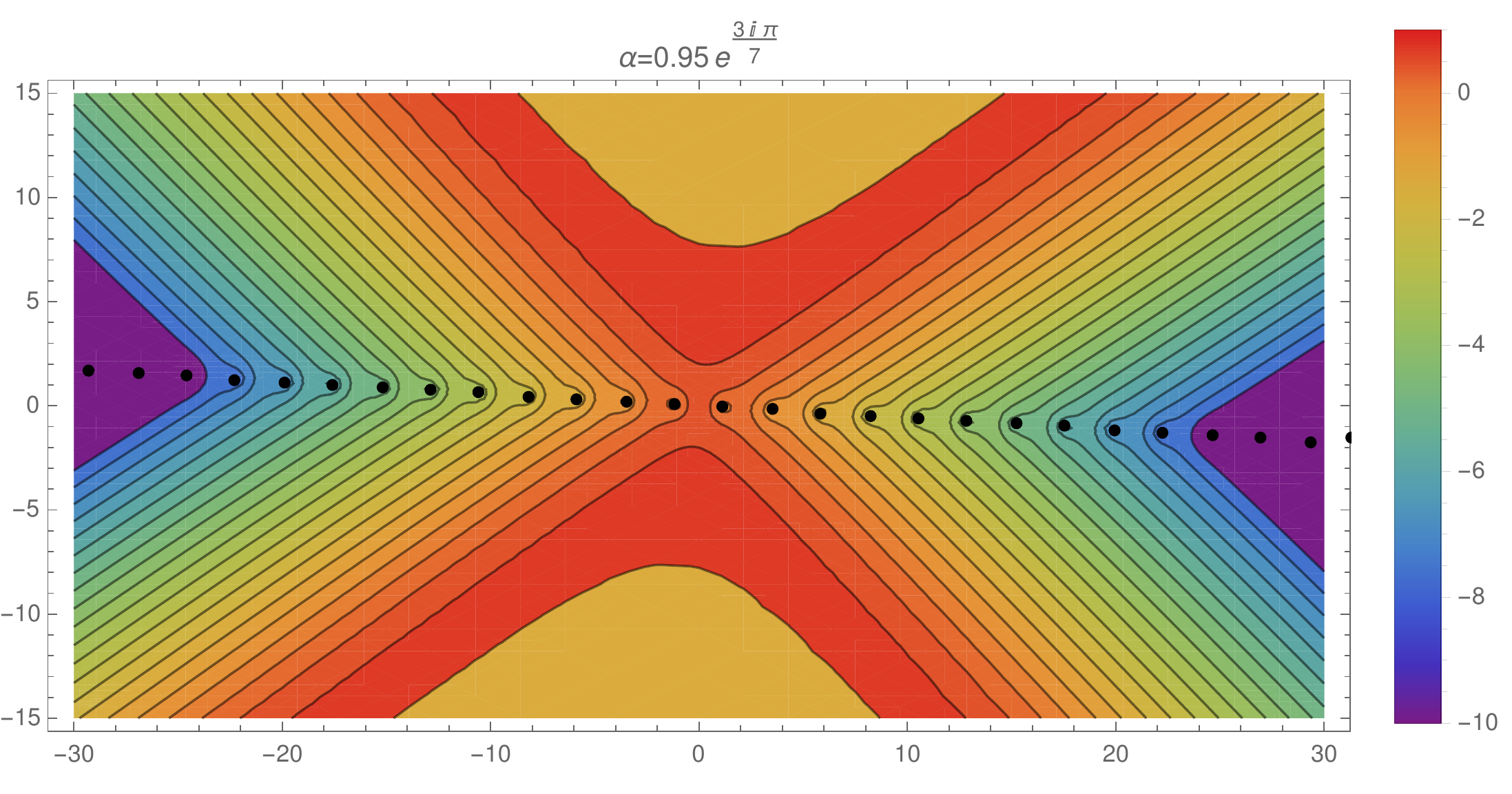}
\qquad
\includegraphics[width=0.36 \textwidth]{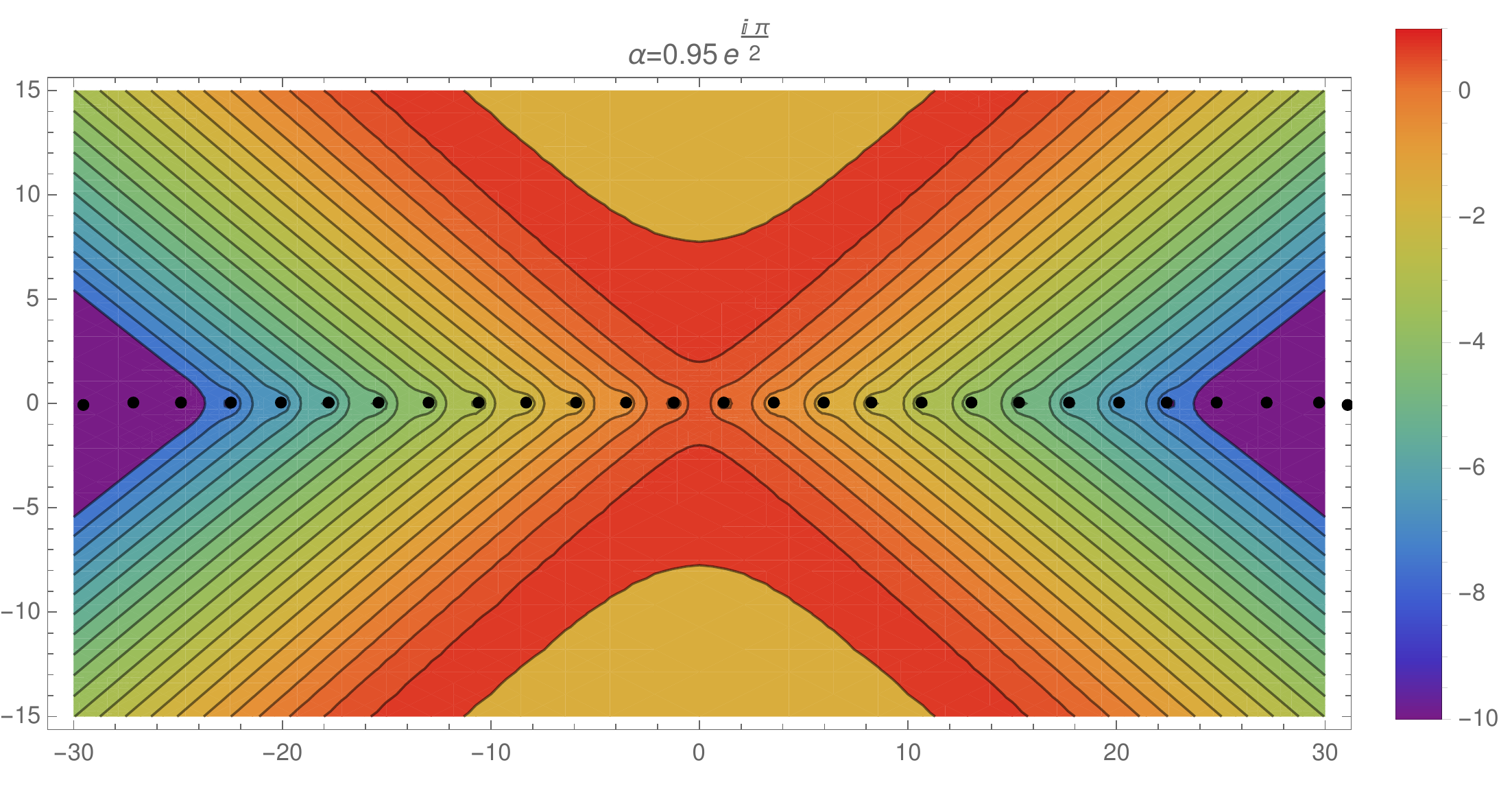}
\caption{Pseudospectra of $J(\alpha)$ with $\alpha$'s lying on the circle with $|\alpha|=0.95$.}
\label{fig:ps_circle}
\end{figure}
\begin{figure}[htb!]
\includegraphics[width=0.36 \textwidth]{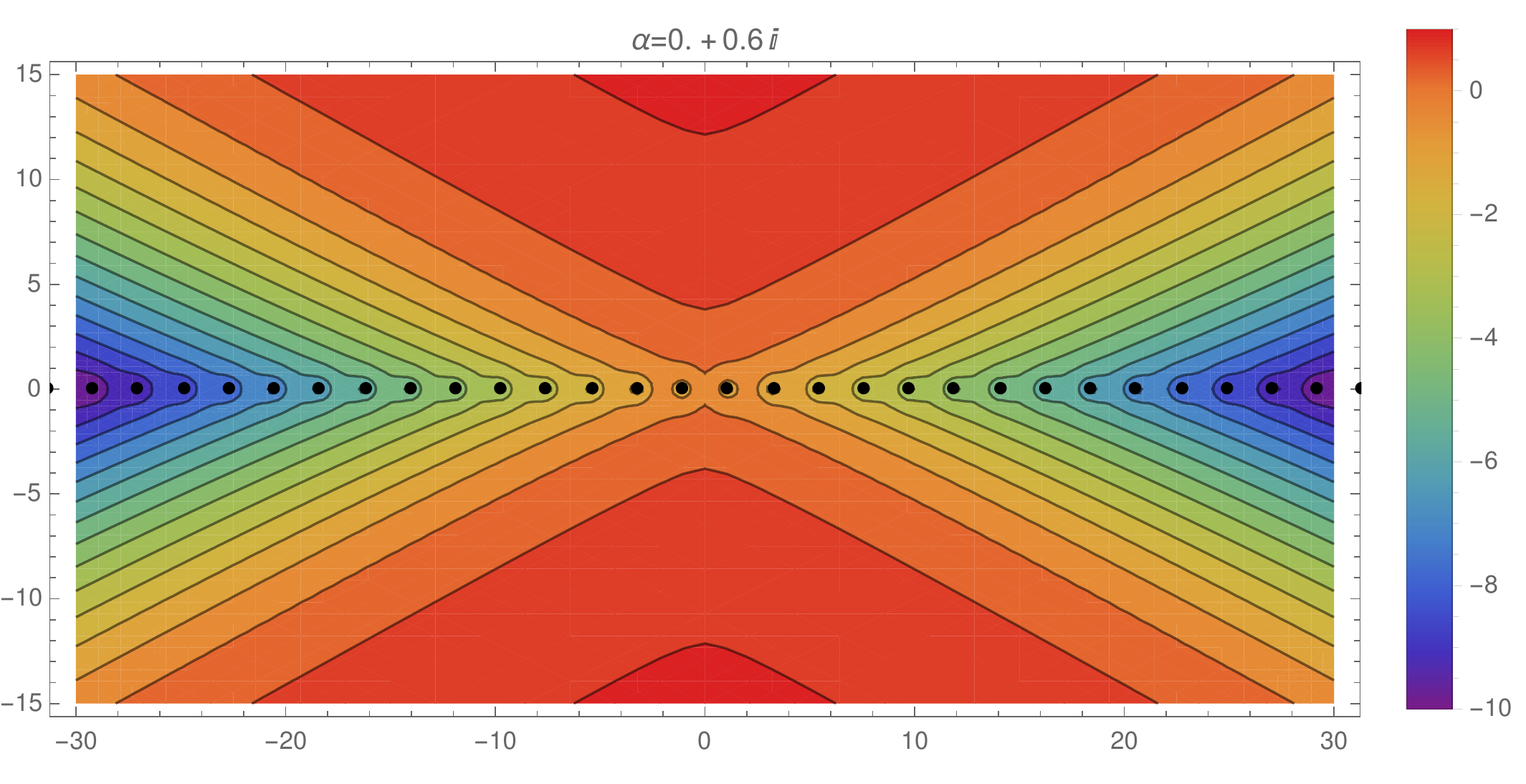}
\qquad
\includegraphics[width=0.36 \textwidth]{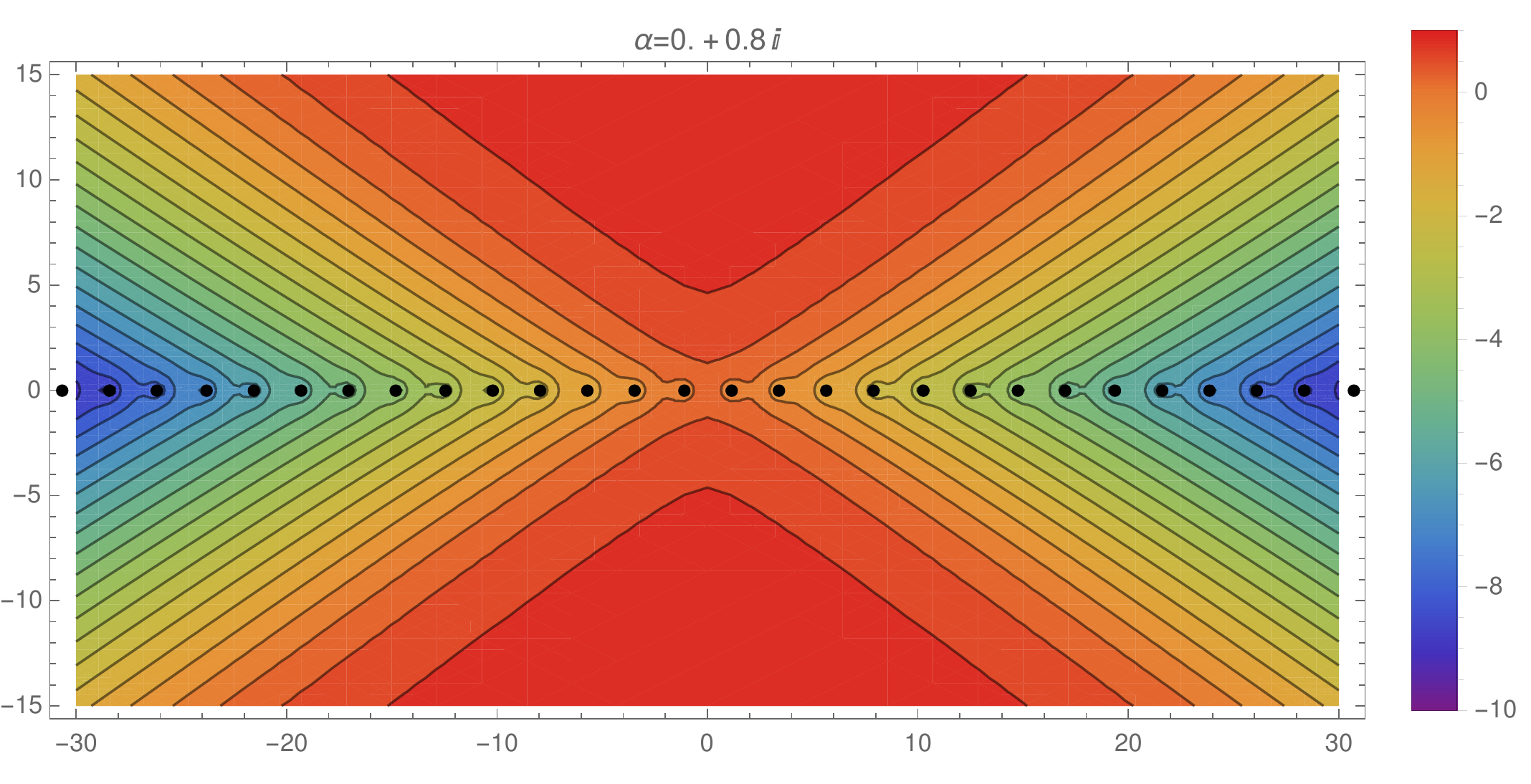}
\\
\includegraphics[width=0.36 \textwidth]{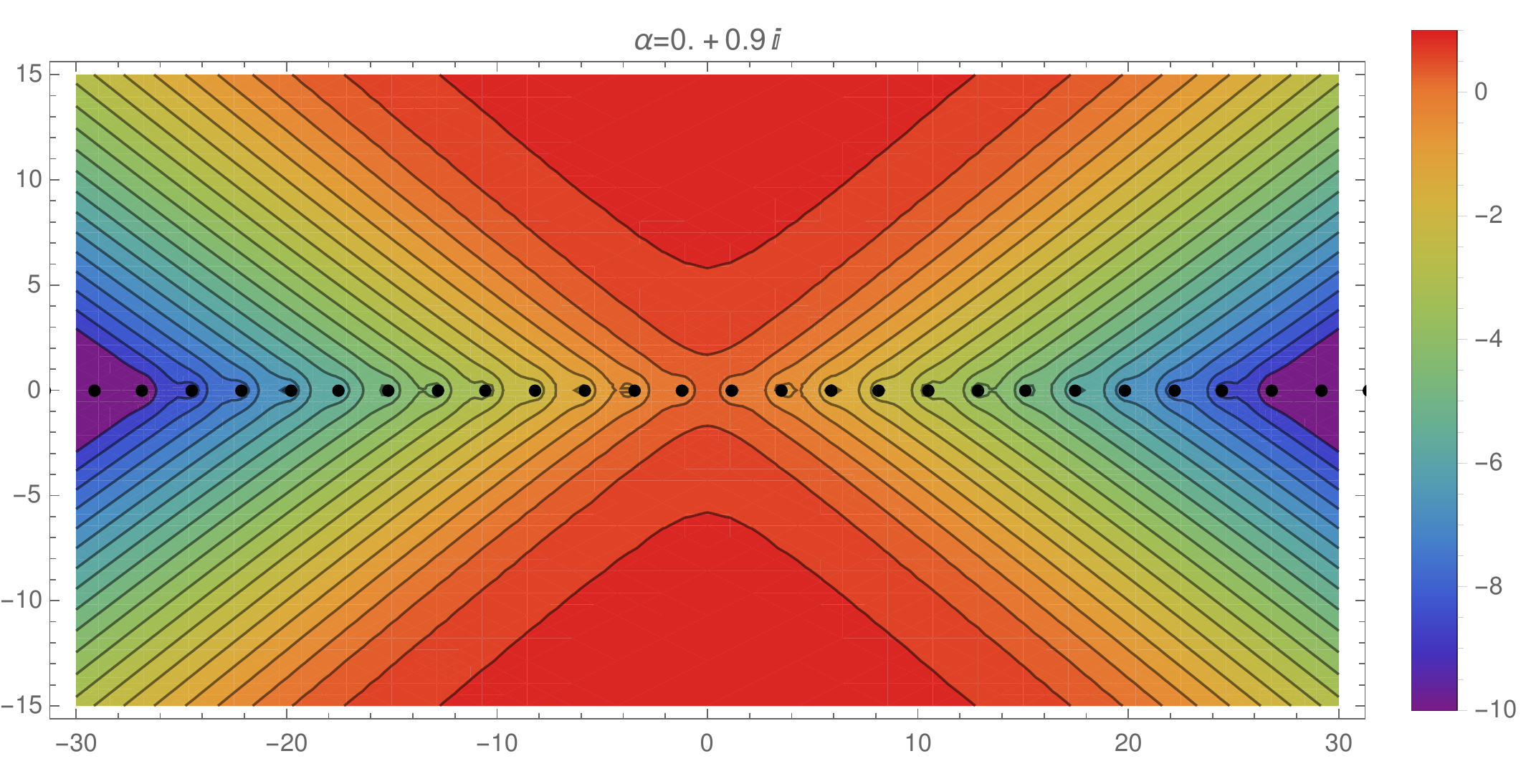}
\qquad
\includegraphics[width=0.36 \textwidth]{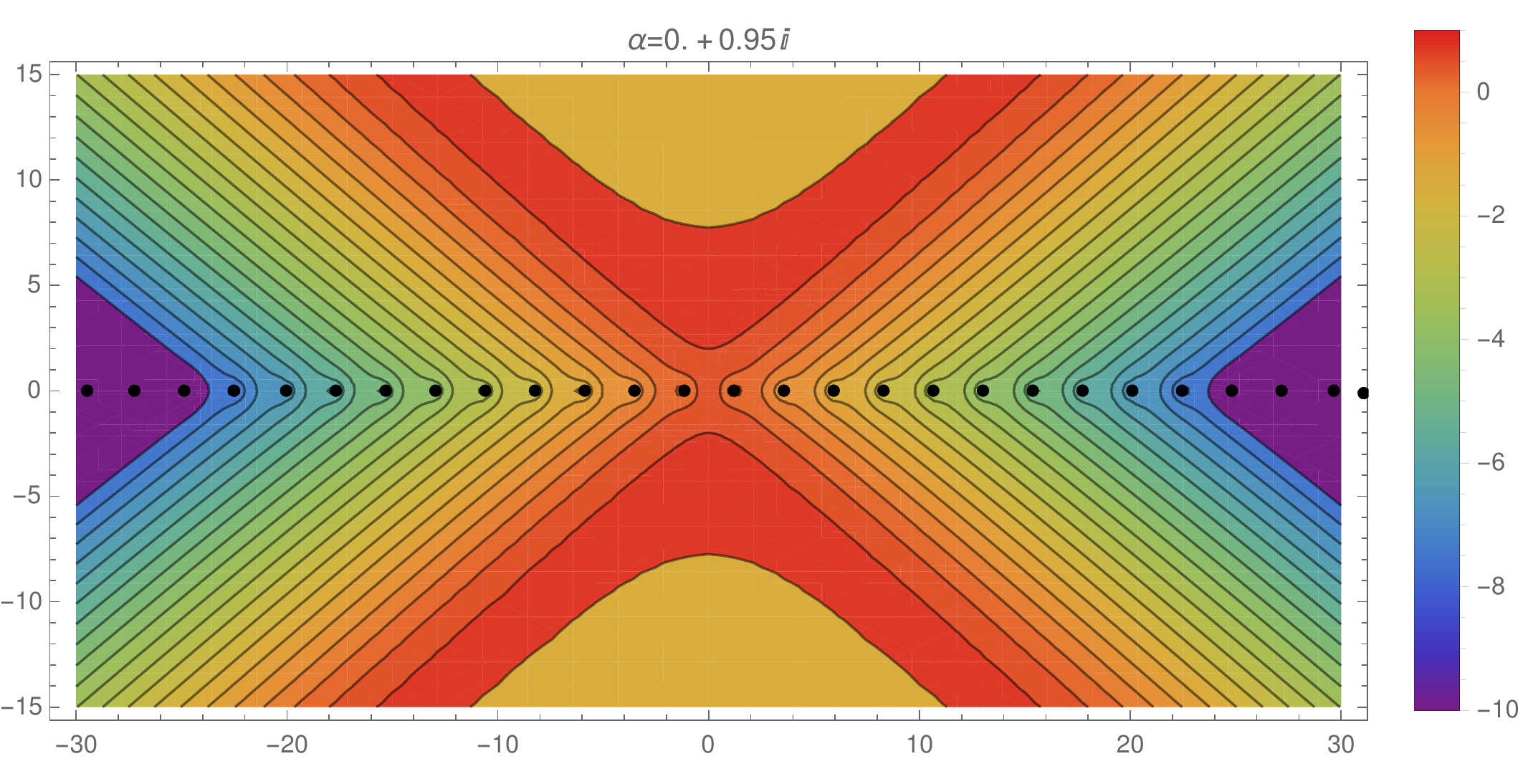}
\caption{Pseudospectra of $J(\alpha)$ with purely imaginary $\alpha$'s approaching $\ii$.}
\label{fig:ps_im}
\end{figure}


{\footnotesize
\bibliographystyle{acm}
\bibliography{references}
}

\end{document}